\documentclass[11pt]{article}

\title{Limit theory for the empirical extremogram of random fields}

\author{Sven Buhl
\thanks{Center for Mathematical Sciences, Technical University of Munich, 85748 Garching, Germany, http://www.statistics.ma.tum.de, Email: sven.buhl@tum.de, cklu@tum.de}
\and Claudia Kl\"uppelberg\footnotemark[1]
}

\usepackage{color}
\usepackage{amssymb}
\usepackage{amsfonts}
\usepackage{amsmath}
\usepackage{amsthm}
\usepackage[numbers]{natbib}
\usepackage{comment}
\usepackage[small,bf,figurewithin=section]{caption}
\usepackage{graphicx}
\usepackage{graphics}
\usepackage{nicefrac}
\usepackage{multicol}
\usepackage{enumerate}
\usepackage{dsfont}
\usepackage{bbm}
\usepackage{url}
\usepackage{capt-of}
\usepackage[english]{babel}
\usepackage[]{inputenc}
\usepackage{accents}

\newlength{\dtildeheight}

\numberwithin{equation}{section}

\newtheorem{theorem}{Theorem}[section]
\newtheorem{lemma}[theorem]{Lemma}
\newtheorem{proposition}[theorem]{Proposition}
\newtheorem{definition}[theorem]{Definition}
\newtheorem{corollary}[theorem]{Corollary}

\newtheorem{fig}[theorem]{Figure}
\newtheorem{example}[theorem]{Example}
\newtheorem{remark}[theorem]{Remark}

\newcommand{\trans}{{}^{^{\intercal}}}

\newcommand{\bthe}{\begin{theorem}}
\newcommand{\ethe}{\end{theorem}}

\newcommand{\ben}{\begin{enumerate}}
\newcommand{\een}{\end{enumerate}}

\newcommand{\beq}{\begin{equation}}
\newcommand{\eeq}{\end{equation}}

\newcommand{\ble}{\begin{lemma}}
\newcommand{\ele}{\end{lemma}}

\newcommand{\bde}{\begin{definition}}
\newcommand{\ede}{\end{definition}}

\newcommand{\bco}{\begin{corollary}}
\newcommand{\eco}{\end{corollary}}

\newcommand{\bpr}{\begin{proposition}}
\newcommand{\epr}{\end{proposition}}

\newcommand{\bproof}{\begin{proof}}
\newcommand{\eproof}{\end{proof}}

\newcommand{\bexam}{\begin{example}\rm}
\newcommand{\eexam}{\halmos\end{example}}

\newcommand{\brem}{\begin{remark}\rm}
\newcommand{\erem}{\halmos\end{remark}}

\newcommand{\bfi}{\begin{fig}}
\newcommand{\efi}{\end{fig}}

\newcommand{\btab}{\begin{tab}}
\newcommand{\etab}{\end{tab}}

\newcommand{\beao}{\begin{eqnarray*}}
\newcommand{\eeao}{\end{eqnarray*}\noindent}

\newcommand{\beam}{\begin{eqnarray}}
\newcommand{\eeam}{\end{eqnarray}\noindent}

\newcommand{\barr}{\begin{array}}
\newcommand{\earr}{\end{array}}

\newcommand{\bdis}{\begin{displaymath}}
\newcommand{\edis}{\end{displaymath}\noindent}
\newcommand{\bs}{\boldsymbol}

\newcommand{\bbn}{\mathbb{N}}
\newcommand{\bbz}{\mathbb{Z}}
\newcommand{\bbr}{\mathbb{R}}
\newcommand{\N}{\mathbb{N}}
\newcommand{\Z}{\mathbb{Z}}
\newcommand{\R}{\mathbb{R}}
\newcommand{\E}{\mathbb{E}}

\def\cals{{\mathcal{S}}}

\def\calh{{\mathcal{H}}}

\newcommand{\stp}{\stackrel{P}{\rightarrow}}
\newcommand{\std}{\stackrel{d}{\rightarrow}}

\newcommand{\stv}{\stackrel{v}{\rightarrow}}

\newcommand{\nto}{n\to\infty}
\newcommand{\kto}{k\to\infty}

\newcommand{\im}{\mathrm{i}}

\newcommand{\ga}{{\gamma}}
 
\newcommand{\si}{{\sigma}}

\newcommand{\eps}{\varepsilon}
\newcommand{\var}{\mathbb{V}{\rm ar}}
\newcommand{\cov}{\mathbb{C}{\rm ov}}

\newcommand{\ov}{\overline}
\newcommand{\wh}{\widehat}

\newcommand{\one}{\mathbbmss{1}_}

\newcommand{\halmos}{\quad\hfill\mbox{$\Box$}}  

\allowdisplaybreaks[4]

\begin{document}

\maketitle

\begin{abstract}
Regularly varying stochastic processes are able to model extremal dependence between process values at locations in random fields. 
We investigate the empirical extremogram as an estimator of dependence in the extremes.
We provide conditions to ensure asymptotic normality of the empirical extremogram centred by a pre-asymptotic version.
For max-stable processes with Fr{\'e}chet margins we provide conditions such that the empirical extremogram centred by its true version is asymptotically normal.
The results of this paper apply to a variety of spatial and space-time processes, and to time series models.
We apply our results to max-moving average processes and Brown-Resnick processes.
\end{abstract}

\noindent
\begin{tabbing}
{\em AMS 2010 Subject Classifications:} \= primary:\,\,\, \, \,\,
60F05, 
60G70, 
62G32 
\\
\> secondary: \,\,\,
37A25, 
62M30 
\end{tabbing}


\noindent
{\em Keywords:}
Brown-Resnick process, empirical extremogram, extremogram, max-moving average process, max-stable process, random field, spatial CLT, spatial mixing.

\vspace{0.5cm}
\section{Introduction}\label{Introduction}
 
The extremogram measures extremal dependence in a strictly stationary regularly varying stochastic process and can hence be seen as a correlogram for extreme events.
It was introduced in \citet{Davis2} for time series (also in \citet{Fasen}), and they show consistency and asymptotic normality of an empirical extremogram  under weak mixing conditions.
 \citet{DMZ} give a profound review of the  estimation theory for time series with various examples.
 For a discussion of the role of the extremogram in dependence modelling of extremes we refer again to \cite{Davis2}.
As it is spelt out there, it is the covariance function of indicator functions of exceedance events in an asymptotic sense.
Also in that paper classical mixing conditions as presented in \citet{Ibragimov}, on which we rely in our work, are compared to the extreme mixing conditions $D$ and $D'$ often used in extreme value theory (cf. \citet{EKM}, Section~4.4, and \citet{Leadbetter}, Sections~3.1 and~3.2).
 
The extremogram and its empirical estimate have been formulated for spatial $d$-dimensional stochastic processes by \citet{Cho} and for space-time processes in \citet{Steinkohl3} and \citet{steinkohlphd}, when observed on a regular grid.
The  extremogram is defined for strictly stationary regularly varying stochastic processes, where all finite-dimensional distributions are in the maximum domain of attraction of some Fr\'echet distribution.
Among other results, based on the seminal paper \cite{Bolthausen} by Bolthausen, \cite{Cho} prove a CLT for the empirical extremogram sampled at different spatial lags, centred by the so-called pre-asymptotic extremogram. Such results also compare with a CLT for sample space-time covariance estimators derived in \citet{Genton}, also based on \cite{Bolthausen}.
 
The pre-asymptotic extremogram can be replaced in the CLT by the true one, if a certain bias condition is satisfied; in particular, the difference between the pre-asymptotic and the true extremogram must vanish with the same rate as the one given in the CLT.
However, for many processes the assumptions required in \cite{Cho} are too restrictive to satisfy this bias condition. 
We explain this in detail and present two models which exactly fall into this class; the max-moving average process and the Brown-Resnick process. 
These two processes are max-stable with Fr\'echet margins.

In this paper, we prove a CLT for the empirical extremogram centred by the pre-asymptotic extremogram for strictly stationary regularly varying stochastic processes, which relies on weaker conditions than the CLT stated in \cite{Cho}. 
Our proof also partly relies on Bolthausen's CLT for spatial processes in \cite{Bolthausen}; however, we make important modifications so that the bias condition mentioned above can be satisfied, and thus a CLT for the empirical extremogram centred by the true one for many more processes becomes possible. 
The proof is based on a big block/small block argument, similarly to \cite{Davis2}.

Our interest is of course in a CLT centred by the true extremogram, and whether such a CLT is possible depends on the specific regularly varying process.
If the process has finite-dimensional max-stable distributions, in our framework equivalent to having finite-dimensional Fr\'echet distributions, we can give conditions such that a CLT of that kind is possible. 
Here we need the weaker mixing conditions of our version of Bolthausen's CLT compared to \cite{Cho}.
Furthermore, under conditions such that a CLT centred by the true extremogram is not possible, a bias-corrected estimator can be defined, which we do in the accompanying paper  \citet{Steinkohl3} for the Brown-Resnick process.
 
Our paper is organised as follows.
In Section~2 we present the general model class of strictly stationary regularly varying processes in $\R^d$ for $d\in\N$.
We also define here the extremogram for such processes.
In Section~3 we define the empirical extremogram based on grid observations, and also the pre-asymptotic extremogram.
Section 4 is devoted to the CLT for the empirical extremogram centred by the pre-asymptotic extremogram and to our examples of max-stable spatial processes; max-moving average processes and Brown-Resnick processes. 
We discuss in detail the problem of a CLT for the empirical extremogram and compare our new conditions for the CLT to hold with those in previous work, particularly with those given in \citet{Cho}. 
For processes with Fr\'echet margins we prove a CLT for the empirical extremogram centred by the true extremogram.
The proof of the CLT is given in Section 5.

\section{Regularly varying spatial processes} \label{S2}

As a natural model class in extreme value theory we consider strictly stationary regularly varying processes $\{X(\bs s): \bs s \in \mathbb{R}^d\}$ for $d\in\N$, where all finite-dimensional distributions are regularly varying (cf. \citet{HL} for definitions and results in a general framework and \citet{Resnick3} for details about multivariate regular variation).
As a prerequisite, we define for every finite set $\mathcal{I} \subset \mathbb{R}^d$ 
the vector 
$$X_{\mathcal{I}}:= (X(\bs s) : \bs s\in \mathcal{I})\trans.$$
Throughout we assume that $X_{\mathcal I}$ inherits the strict stationarity from $\{X(\bs{s}): \bs{s} \in \mathbb{R}^d\}$, which is guaranteed, if we consider lagged vectors of $X_{\mathcal I}$.
Furthermore, $|\mathcal{I}|$ denotes the cardinality of $\mathcal{I}$. 
As usual, $f(n)\sim g(n)$ as $\nto$ means that $\lim_{\nto}\frac{f(n)}{g(n)} =1$.

\begin{definition}[Regularly varying process]
A strictly stationary stochastic process $\{X(\bs{s}): \bs{s} \in \mathbb{R}^d\}$ is called \textnormal{regularly varying}, if there exists some normalising sequence $0<a_n \rightarrow \infty$ such that $\mathbb{P}(|X(\bs 0)|>a_n) \sim n^{-d}$ as $n \rightarrow \infty$
and for every  finite  set $\mathcal{I} \subset \mathbb{R}^d$ 
\begin{align} \label{regvar}
n^d\mathbb{P}\Big(\frac{X_{\mathcal{I}}}{a_n} \in \cdot\Big) \stv \mu_{\mathcal{I}}(\cdot), \quad n \rightarrow \infty,
\end{align}
for some non-null  Radon measure $\mu_{\mathcal{I}}$ on the Borel sets in $\overline{\mathbb{R}}^{|\mathcal{I}|}\backslash\{\bs 0\}$, where $\ov{\mathbb{R}}=\mathbb{R} \cup \{-\infty,\infty\}$. In that case, 
$$\mu_{\mathcal{I}}(xC)=x^{-\beta}\mu_{\mathcal{I}}(C), \quad x>0,$$ 
for every Borel set $C\subset\overline{\mathbb{R}}^{|\mathcal{I}|}\backslash\{\bs 0\}$. 
The notation $\stv$ stands for vague convergence, and $\beta>0$ is called the {\em index of regular variation}.
\end{definition}

For every $\bs s \in \mathbb{R}^d$ and $\mathcal{I}=\{\bs s\}$ we set $\mu_{\{\bs s\}}(\cdot)=\mu_{\{\bs 0\}}(\cdot)=:\mu(\cdot),$ which is justified by stationarity.

The focus of the present paper is on the extremogram, defined for values in $\R^d$ as follows.

\begin{definition}[Extremogram] \label{DefExtremo}
Let $\{X(\bs{s}): \bs{s} \in \mathbb{R}^d\}$ be a strictly stationary regularly varying process in $\R^d$. For two $\mu$-continuous Borel sets $A$ and $B$  in $\overline{\mathbb{R}}\backslash\{0\}$ (i.e., $\mu(\partial A)=\mu(\partial B)=0$) such that $\mu(A) >0$, the \textnormal{extremogram} is defined as
\begin{align} \label{extremo}
\rho_{AB}(\bs{h})=\lim_{n \rightarrow \infty} \frac{\mathbb{P}(X(\bs{0})/a_n \in A,X(\bs{h})/a_n \in B)}{\mathbb{P}(X(\bs{0})/a_n \in A)}, \quad \bs{h} \in \mathbb{R}^d.
\end{align} 
\end{definition}

Our goal is to estimate the extremogram for arbitrary strictly stationary regularly varying processes by its empirical version and prove asymptotic properties like consistency and asymptotic normality.
Such results also allow for semiparametric estimation in a parametric spatial or space-time model as presented in \citet{Steinkohl3}. 

Analogous asymptotic results for the empirical extremogram of time series have been shown in \citet{Davis2} and of $d$-dimensional random fields in \citet{Cho}. 
However, in certain situations, for example in the case of the Brown-Resnick process~\eqref{limitfield}, the rates obtained in \cite{Cho} are too crude to allow for a CLT. 
We apply a small block/large block argument in space (similarly to \cite{Davis2} for time series), which leads to more precise rates in the CLT. 
Arguments of our proof are based on spatial mixing conditions, and rely on general results of \citet{Bolthausen} and \citet{Ibragimov}.

\section{Large sample properties of the spatial empirical extremogram}\label{S3}

The estimation of the extremogram is based on data observed on 
$$\cals_n =\left\{1,\ldots,n\right\}^d = \{\bs{s}_j: j=1,\dots,n^d\}, \mbox{ the regular grid of side length $n$.}$$
Let $\|\cdot\|$ be an arbitrary norm on $\R^d$.
Define the following quantities for $\ga >0$:
\begin{align}
&B(\bs{0},\ga) = \big\{\bs{s}\in \mathbb{Z}^d : \|\bs s\| \leq \ga\big\},  \nonumber\\
&B(\bs{s},\ga) = \big\{\bs{s'}\in \mathbb{Z}^d : \|\bs s-\bs s'\| \leq \ga\big\}=\bs s + B(\bs 0,\ga), \label{notapp}\\
&\mathcal{H} \subseteq \{\bs h=\bs s-\bs s': \bs s, \bs s' \in \cals_n\}\cap B(\bs 0,\ga), \,\text{a finite set of observed lags.}\label{lags}
\end{align}
We further define the vectorized process $\{\bs{Y}(\bs{s}): \bs{s}\in\bbr^d\}$ by 
$$\bs{Y}(\bs s):=X_{B(\bs s,\gamma)};$$
 i.e.,  $\bs Y(\bs s)$ is the vector of values of $X$ with indices in $B(\bs s,\ga)$ as defined in \eqref{notapp}. 
 We introduce the balls $B(\bs 0,\ga)$ in order to express events like $\{X(\bs s) \in A, X(\bs s+\bs h) \in B\}$ or $\{X(\bs s) \in A\}$ for $\bs s \in \bbr^d$ and $\bs h \in \mathcal{H} \subseteq B(\bs 0,\ga)$ as well as Borel sets $A, B$ in $\ov{\bbr} \backslash\{0\}$ through events $\{\bs Y(\bs s) \in C\}$ for appropriately chosen Borel sets $C$ in $\ov{\bbr}^{|B(\bs 0,\ga)|} \backslash\{\bs 0\}$. 
 This notation simplifies the presentation of the proofs of consistency and asymptotic normality considerably.

For $j \in \mathbb{N}$ let $\bs e_j$ be the $j$-th unit vector in $\mathbb{R}^d$. 
The choice of a regular grid $\cals_n$ can be extended to arbitrary observation sets provided that 
they increase to $\Z^d$ and have boundaries {$\partial\cals_n:=\{\bs s \in \cals_n: \exists \, \bs z \in \mathbb{Z}^d\setminus \cals_n \mbox{ and } j \in \mathbb{N} \text{ with } \|\bs z-\bs s\| =\|\bs e_j\| \}$} satisfying $\lim_{\nto} |\partial\cals_n |/|\cals_n| = 0$. 
The natural extension to grids with different side lengths does not involve any additional mathematical difficulty, but notational complexity, since our proofs are based on big/small block arguments common in extreme value statistics, which are much simpler to formulate for a regular grid.

For fixed $n$ and observations on the grid $\cals_n$ there will be points $\bs s \in \cals_n$ near the boundary, such that not all components of $\bs Y(\bs s)$ can be observed.
However,  since we investigate asymptotic properties for $\cals_n$ and the boundary points become negligible, this is irrelevant for our results and we suppress such technical details.
As a consequence, our results apply to spatial as well as time series observations, thus include the frameworks considered in \citet{Cho} and \citet{Davis2}.
 
We shall also need the following relations and definitions, where the limits exist by regular variation of $\{X(\bs s): \bs s \in \mathbb{R}^d\}$. 
Let $C$ be a $\mu_{B(\bs 0, \ga)}$-continuous Borel set in $\ov\R^{|B(\bs 0,\ga)|}\backslash\{\bs 0\}$ and $C\times D$ a $\tau_{{B(\bs 0,\gamma) \times B(\bs h,\gamma)}}$-continuous Borel set in the product space, where we define
\beam
\mu_{B(\bs 0, \ga)}(C)  &:=& \lim_{\nto} n^d\mathbb{P}\Big(\frac{\bs Y(\bs 0)}{a_{n}} \in C\Big)\label{B6},\\
\tau_{{B(\bs 0,\gamma) \times B(\bs h,\gamma)}} (C \times D) 
&:=& \lim_{\nto} n^d\mathbb{P}\Big(\frac{\bs Y(\bs 0)}{a_{n}} \in C, \frac{\bs Y(\bs h)}{a_{n}} \in D\Big).\label{B7}
\eeam
We enumerate the lags in $\mathcal{H}$ by $\mathcal{H}=\{\bs h_1,\ldots, \bs h_p\}.$
Following ideas of \citet{Davis2} (also used in \citet{Cho}) we define {$\mu_{B(\bs 0, \ga)}$-continuous Borel sets $D_1,\ldots,D_{p},D_{p+1}$ in $\ov\R^{|B(\bs 0,\ga)|}\backslash\{\bs 0\}$} by the property 
\beam\label{Di}
\{\bs Y(\bs s)\in D_i\} = \{X(\bs{s})\in A,X(\bs{s}+\bs h_i)\in B\}
\eeam
for $i=1,\ldots,p$, and $\{\bs Y(\bs s)\in D_{p+1}\} = \{X(\bs s)\in A\}$.
Note in particular that,  by the relation between $\{\bs Y(\bs s): \bs s \in\R^d\}$ and $\{X(\bs s): \bs s \in\R^d\}$ and  regular variation, for every $\mu$-continuous Borel set $A$ in $\ov\R\setminus\{0\}$,
$$\mu_{B(\bs 0,\ga)}(D_{p+1})=\lim_{\nto} n^d \mathbb{P}\Big(\frac{{\bs Y} (\bs 0)}{a_n} \in D_{p+1}\Big)
=\lim_{\nto} n^d \mathbb{P}\Big(\frac{X(\bs 0)}{a_n} \in A\Big)=\mu(A).$$

The  extremogram can be estimated from data by the following empirical version.

\begin{definition}[Empirical extremogram] 
Let $\{X(\bs{s}): \bs{s} \in \mathbb{R}^d\}$ be a strictly stationary regularly varying  process in $\R^d$, observed on $\cals_n$, and set {$\cals_n(\bs h):=\{\bs s \in \cals_n: \bs s+\bs h \in \cals_n\}$} for $\bs h \in \mathcal{H}$. 
Let $A$ and $B$ be $\mu$-continuous Borel sets  in $\overline{\mathbb{R}}\backslash\{0\}$ such that $\mu(A)>0$. 
For a sequence $m=m_n \rightarrow \infty$ and $m_n=o(n)$ as $\nto$, the \textnormal{empirical extremogram} is defined {for $\bs h\in\calh$} as 
\begin{align} \label{EmpEst}
\widehat{\rho}_{AB,m_n}(\bs h):=\frac{\dfrac1{|\cals_n(\bs h)|}\sum\limits_{\bs s \in \cals_n(\bs h)} \mathbbmss{1}_{\{X(\bs s)/a_m \in A, X(\bs s+\bs h)/a_m \in B\}}}
{\dfrac1{|\cals_n|}\sum\limits_{\bs s \in \cals_n} \mathbbmss{1}_{\{ X(\bs s)/a_m \in A\}}}.
\end{align}
\end{definition}

The following pre-asymptotic extremogram plays an important role when proving asymptotic normality of the empirical  extremogram~\eqref{EmpEst}.

\begin{definition}[Pre-asymptotic extremogram]
Let $\{X(\bs{s}): \bs{s} \in \mathbb{R}^d\}$ be a strictly stationary regularly varying process in $\R^d$. 
Let $A$ and $B$ be $\mu$-continuous Borel sets  in $\overline{\mathbb{R}}\backslash\{0\}$ such that $\mu(A)>0$. 
For a sequence  $m=m_n \rightarrow \infty$ and $m_n=o(n)$ as $\nto$, the \textnormal{pre-asymptotic extremogram} is defined as 
\begin{equation}\label{preasymptotic}
\rho_{AB,m_n}(\bs h) = \frac{P\left(X(\bs{0})/a_m \in A, X(\bs{h})/a_m\in B\right)}{\mathbb{P}(X(\bs{0})/a_m\in A)}.
\end{equation}
\end{definition}

The next section is devoted to the asymptotic properties of the empirical extremogram and the inherent bias-variance problem with its solution.

\section{Consistency and CLT for the empirical extremogram}

In this section we derive relevant asymptotic properties of the empirical extremogram.
First we establish large sample properties of the empirical estimator of $\mu_{B(\bs 0,\ga)}(C)$.
Based on these results, the asymptotic normality is established.

Throughout this section we assume that $\{X(\bs s): \bs s \in \mathbb{R}^d\}$ is a strictly stationary regularly varying  process in $\R^d$, observed on $\cals_n.$

We need the concept of $\alpha$-mixing for such processes; see e.g.  \citet{Bolthausen} or \citet{Doukhan}.

\begin{definition}[$\alpha$-mixing coefficients and $\alpha$-mixing] \label{mixing}
Consider a strictly stationary random field $\left\{X(\bs{s}): \bs{s}\in \bbr^d\right\}$ and let $d(\cdot,\cdot)$ be some metric induced by a norm $\|\cdot\|$ on $\mathbb{R}^d$.
For $\Lambda_1, \Lambda_2 \subset \mathbb{Z}^d$ set 
\begin{align*}
d(\Lambda_1,\Lambda_2) := \inf\left\{\|\bs{s}_1-\bs{s}_2\|: \ \bs{s}_1 \in \Lambda_1, \bs{s}_2 \in \Lambda_2\right\}.
\end{align*}
Further, for $i=1,2$ denote $\mathcal{F}_{\Lambda_i}= \sigma\left\{X(\bs{s}):  \bs{s}\in \Lambda_i\right\}$ the $\sigma$-algebra generated by $\{X(\bs{s}): \ \bs{s}\in \Lambda_i\}$.
\begin{enumerate}[(i)]
\item 
The {\em $\alpha$-mixing coefficients} are defined for $k,\ell \in \mathbb{N} \cup \{\infty\}$ and $r \geq 0$ by
\begin{align}\label{alphaBolt}
\alpha_{k,\ell}(r) := & \sup\{|\mathbb{P}(A_1\cap A_2) - \mathbb{P}(A_1)\mathbb{P}(A_2)|: \nonumber\\
 & \quad\quad A_i \in \mathcal{F}_{\Lambda_i}, |\Lambda_1|\leq k, |\Lambda_2|\leq \ell, d(\Lambda_1,\Lambda_2) \geq r\}.
\end{align}
\item 
The random field is called {\em $\alpha$-mixing}, if $\alpha_{k,l}(r) \to 0$ as $r\to\infty$ for all $k,\ell \in \bbn$.
\end{enumerate}
\end{definition}

In what follows we have to control the dependence between vector processes $\{\bs Y(\bs s)=X_{B(\bs s,\ga)}: s \in \Lambda_1'\}$ and $\{\bs Y(\bs s)=X_{B(\bs s,\ga)}: \bs s \in \Lambda_2'\}$ for subsets $\Lambda_i' \subset \mathbb{Z}^d$ with cardinalities $|\Lambda_1'| \leq k$ and $|\Lambda_2'| \leq \ell.$ 
In the context of Definition~\ref{mixing}, this means that the $\Lambda_i$ in \eqref{alphaBolt} are unions of balls $\Lambda_i=\cup_{\bs s \in \Lambda_i'} B(\bs s,\ga)$. 
Since $\ga>0$ is some predetermined finite constant independent of $n$, we keep notation simple by redefining the $\alpha$-mixing coefficients with respect to the vector processes as
\begin{align}\label{alpha_balls}
\alpha_{k,\ell}(r) := & \sup\{|\mathbb{P}(A_1\cap A_2) - \mathbb{P}(A_1)\mathbb{P}(A_2)|: \nonumber\\
& \quad\quad A_i \in \mathcal{F}_{\Lambda_i},\,\,\Lambda_i=\cup_{\bs s \in \Lambda_i'} B(\bs s,\ga),\,\, |\Lambda_1'|\leq k, |\Lambda_2'|\leq \ell, d(\Lambda_1',\Lambda_2') \geq r\}.
\end{align}

\bthe\label{stasyn}
Let $\{X(\bs s): \bs s \in \mathbb{R}^d\}$ be a strictly stationary regularly varying  process, observed on $\cals_n$ and let $\mathcal{H}=\{\bs h_1,\ldots, \bs h_p\}$ be a finite set of lags in $\mathbb{Z}^d$ satisfying $\mathcal{H} \subseteq B(\bs 0, \ga)$ for some $\ga>0$.
Suppose that the following conditions are satisfied: 
\begin{enumerate}[(M1)]
\item 
$\{X(\bs{s}): \bs{s}\in\bbr^d\}$ is $\alpha$-mixing with $\alpha$-mixing coefficients $\alpha_{k,\ell}(r)$ defined in \eqref{alphaBolt}.
\end{enumerate}
There exist sequences $m=m_n, r=r_\nto $ with $m_n/n \to 0$ and $r_n/m_n \to 0$ as $\nto$ such that the following hold:
\begin{enumerate}[(M1)]
\setcounter{enumi}{1}
\item 
$m_n^{2}r_n^{2}/n \to 0$.
\label{cond0}
\item 
\label{cond1}
For all $\epsilon>0$: 
\begin{align*}
&\lim\limits_{\kto } \limsup\limits_{\nto }
\sum\limits_{\bs h \in \mathbb{Z}^{d} :  k< \Vert{\bs h}\Vert \leq r_n}
 m_n^d \\
& \quad \quad\quad  \mathbb{P}\Big(\max\limits_{\bs s \in B(\bs 0,\gamma)} 
|X(\bs s)|>\epsilon a_m, \max\limits_{\bs s' \in B(\bs h,\gamma)} |X(\bs s')|>\epsilon a_m\Big)=0.
\end{align*} 
\item 
\begin{enumerate}[(i)]
\item 
$\lim\limits_{\nto } m_n^d \sum\limits_{\bs h \in \mathbb{Z}^{d} : \Vert{\bs h}\Vert > r_n} \alpha_{1, 1}(\|\bs h\|)=0$, \label{cond2} 
\item 
$\sum\limits_{\bs h \in \mathbb{Z}^{d}} \alpha_{p,q}(\Vert{\bs h}\Vert) < \infty$ \ \mbox{ for } $2\le p+q\le 4$, \label{cond3}
\item 
$\lim\limits_{\nto }  m_n^{d/2} n{^{d/2}} \ \alpha_{1,{n^d}}(r_n)=0$, \label{cond4} 
\end{enumerate}
\end{enumerate}
Then the empirical extremogram $\widehat{\rho}_{AB,m_n}(\bs h)$ for $\bs h \in \mathcal{H}$ as in \eqref{EmpEst}, centred by the pre-asymptotic extremogram in \eqref{preasymptotic}, is asymptotically normal; more precisely,
\begin{equation}
\Big(\frac{n}{m_n}\Big)^{d/2}\big(\widehat{\rho}_{AB,m_n}(\bs h) -\rho_{AB,m_n}(\bs h)\big)_{\bs h\in \mathcal{H}} \std \mathcal{N}(\bs 0,\Pi),\quad \nto,\label{asyspace}
\end{equation}
where
 $\Pi = \mu(A)^{-4} F\Sigma F\trans \in \mathbb{R}^{p \times p}$, and the matrix $\Sigma \in \mathbb{R}^{(p+1) \times (p+1)}$ has
 for $1 \leq i,j \leq p+1$ components
\begin{align}
\Sigma_{ii} &= 
\mu_{B(\bs 0, \ga)}(D_i) + \sum_{\bs h \in \mathbb{Z}^d \backslash\{\bs 0\}}\tau_{B(\bs{0},\ga)\times B(\bs h,\ga)}(D_i\times D_i)=:\sigma^2_{B(\bs 0,\gamma)}(D_i), \label{SigmaCoeff1}\\
\Sigma_{ij} &= \mu_{B(\bs 0, \ga)}(D_i\cap D_j) + \sum_{\bs h \in \mathbb{Z}^d \backslash\{\bs 0\}} \tau_{B(\bs{0},\ga)\times B(\bs{h},\ga)}(D_i\times D_j), \quad i \neq j. \label{SigmaCoeff2}
\end{align}
The matrix $F$ consists of a diagonal matrix $F_1$ and a vector $F_2$ in the last column:
\begin{align}
F &= [F_1,F_2]\quad\mbox{with} \label{matrixF}\\
 F_1 &= \textnormal{diag}(\mu(A)) \in \mathbb{R}^{p \times p}, \quad F_2 = (-\mu_{B(\bs 0, \ga)}(D_1),\ldots,-\mu_{B(\bs 0, \ga)}(D_{p}))\trans.\nonumber
\end{align}
\ethe

\brem\label{rem4.3}
(i) \, In the proof (given in Section~5) we use a big block/small block argument. For simplicity we assume that $n/m_n$ is an integer and subdivide $\cals_n$ into $(n/m_n)^d$ non-overlapping $d$-dimensional blocks with side length $m_n$.
Theorem~1 of \citet{Cho} divides $n^d$ into {$n^d/m_n$} blocks; i.e., their sequence $m_n$ corresponds to our $m_n^d$, so that their assumptions look slightly different.
In our notation, they require that $m_n^{2(d+1)}/n \to 0$ as $\nto$, which is more restrictive than condition (M2) of $m_n^2r_n^2/n \to 0$, and indeed too restrictive for processes such as the max-moving average process and the Brown-Resnick process discussed below.\\[2mm]
(ii) \, If the choice $m_n=n^{\beta_1}$ and $r_n=n^{\beta_2}$ with $0 < \beta_2 < \beta_1 <1$ satisfies conditions (M3) and (M4),
then for $\beta_1 \in (0,1/2)$ and $\beta_2 \in (0,\min\{\beta_1;1/2-\beta_1\})$ the condition (M2) also holds and we obtain the CLT \eqref{asyspace}. 
\erem

The pre-asymptotic extremogram \eqref{preasymptotic} in the central limit theorem can be replaced by the theoretical one \eqref{extremo}, if it converges to the theoretical extremogram with the same convergence rate as the empirical extremogram to the pre-asymptotic extremogram; i.e., if 
\begin{equation}\label{condition5}
\Big(\frac{n}{m_n}\Big)^{d/2}\left(\rho_{AB,m_n}(\bs h) - \rho_{AB}(\bs h)\right)_{\bs h \in \mathcal{H}} \to \bs 0,\quad n\to \infty.
\end{equation}
Relation \eqref{condition5} does not hold for every strictly stationary regularly varying spatial process or time series for which \eqref{asyspace} is satisfied. 
Theorem~\ref{CLT_True} states a necessary and sufficient condition for processes with Fr{\'e}chet marginal distributions such that both \eqref{asyspace} and \eqref{condition5} hold. 
For general regularly varying stochastic processes such a result would hold under second order conditions on the finite-dimensional regularly varying distributions of the process, but we do not pursue this topic further.

\begin{theorem}[CLT for processes with Fr{\'e}chet margins]\label{CLT_True}
Let $\{X(\bs s): \bs s \in \bbr^d\}$ be a strictly stationary max-stable process in $\R^d$ with 
standard unit Fr{\'e}chet margins. 
Let $\rho_{AB}$ be its extremogram \eqref{extremo} and $\rho_{AB,m_n}$ the corresponding pre-asymptotic version \eqref{preasymptotic} for sets $A=(a_1,a_2)$ and $B=(b_1,b_2)$ with $0<a_1<a_2 \leq \infty$ and $0 <b_1<b_2 \leq \infty.$  Assume that the process is observed on $\cals_n$ and let $\mathcal{H}=\{\bs h_1,\ldots, \bs h_p\}$ be a finite set of lags in $\mathbb{Z}^d$ satisfying $\mathcal{H} \subseteq B(\bs 0, \ga)$ for some $\ga>0$. 
Furthermore, suppose that conditions (M1)--(M4) of Theorem~\ref{stasyn} hold for appropriately chosen sequences $m_n,r_n\to\infty$. 
\noindent 
Then the limit relation \eqref{condition5} holds if and only if $n/m_n^3 \to 0$ as $\nto$.
In this case we obtain
\begin{equation}
\Big(\frac{n}{m_n}\Big)^{d/2}\big(\widehat{\rho}_{AB,m_n}(\bs h) -\rho_{AB}(\bs h)\big)_{\bs h\in \mathcal{H}} \std \mathcal{N}(\bs 0,\Pi),\quad \nto,\label{asyspace_trueext}
\end{equation}
where $\Pi$ is given in Theorem~\ref{stasyn}.
\end{theorem}

\begin{proof}
First note that, since all finite-dimensional distributions are max-stable distributions with standard unit Fr\'echet margins, they are multivariate regularly varying. 
Let $V_2(\bs h;\cdot,\cdot)$ be the bivariate exponent measure defined through $\mathbb{P}(X(\bs 0) \leq x_1,X(\bs h) \leq x_2)$ $=\exp\{-V_2(\bs h;x_1,x_2)\}$ for $x_1,x_2>0$, cf. \citet{Beirlant}, Section~8.2.2.
From Lemma~\ref{LemmaExtremo_Fr}(b) we know that for $\bs h \in \mathcal{H}$
\begin{align*}
\rho_{AB,m_n}(\bs h)=&\Big[\rho_{AB}(\bs h) 
+ \frac1{2\,m_n^{d}}\ov{V}^2_2(\bs h)\Big](1+o(1)),\quad\nto,
\end{align*}
where $\ov{V}^2_2(\bs h):= \frac{a_1a_2}{a_2-a_1} ( V_2^2(\bs h;a_2,b_2) + V_2^2(\bs h;a_2,b_1) + V_2^2(\bs h;a_1,b_2) +V_2^2(\bs h;a_1,b_1))$ as given in \eqref{preasym_Fr} and
appropriate adaptations for $a_2=\infty$ and/or $b_2=\infty$ given in \eqref{preasym_Fr2}.
Hence, for $\bs h \in \mathcal{H}$,
 \begin{align*}
\Big(\frac{n}{m_n}\Big)^{d/2}\big(\rho_{AB,m_n}(\bs h) - \rho_{AB}(\bs h)\big)  
\sim  \Big(\frac{n}{m_n^{3}}\Big)^{d/2} \frac{\ov{V}^2_2(\bs h)}{2},
 \end{align*}
which converges to 0 if and only if $n/m_n^3 \to 0$.
\end{proof}

\brem\label{rem4.4}
The requirement $n/m_n^3 \to 0$ as $\nto$ needed in Theorem~\ref{CLT_True} contradicts the condition $m_n^{2(d+1)}/n \to 0$ required in \citet{Cho}; thus, under the conditions stated in that paper, only the CLT \eqref{asyspace} centred by the pre-asymptotic extremogram can be proved. 
However, $n/m_n^3 \to 0$ as $\nto$ does not contradict the assumptions of Theorem~\ref{stasyn} above, in particular, the much weaker assumption (M2). \\[2mm]
(ii) \, From Theorem~\ref{CLT_True} in relation to Remark~\ref{rem4.3}~(ii) we conclude that we need to choose $\beta_1>1/3$ in order to satisfy the CLT \eqref{asyspace_trueext}. This is not a contradiction to the conditions of Theorem~\ref{stasyn} and we conclude that for $\beta_1 \in (1/3,1/2)$ and $\beta_2 \in (0,\min\{\beta_1;1/2-\beta_1\}),$ we have 
\beam\label{CLT_Truebeta}
n^{\frac{d}{2}(1-\beta_1)}\big(\widehat{\rho}_{AB,m_n}(\bs h) -\rho_{AB}(\bs h)\big)_{\bs h\in \mathcal{H}} \std \mathcal{N}(\bs 0,\Pi),\quad \nto.
\eeam
\erem

We discuss our findings for two prominent examples.

\bexam[Max-moving average  (MMA) process]\label{ex:MMA} \\
We start with a model for the process $\{X(\bs s): \bs s \in \bbz^d\}$ corresponding to the discrete observation scheme. 
It has been suggested in \citet{Cho} based on a time series model of \citet{DR89}.
Let $Z(\bs s)$ for $\bs s \in \mathbb{Z}^d$ be i.i.d. standard unit Fr{\'e}chet random variables and set 
\begin{align}
X^\star(\bs s):=\max_{\bs z \in \bbz^d} \phi^{\|\bs z\|} Z(\bs s-\bs z), \quad \bs s \in \bbz^d, \label{MMA}
\end{align}
for some $0<\phi<1$. 
Then $\{X^\star(\bs s): \bs s \in \bbz^d\}$ is a stationary \textit{max-moving average} (MMA) process, also considered in equation~(25) of \cite{Cho}. 
As in \cite{Cho} we deduce the following marginal distributions.
The number $N(j)$ of lag vectors $\bs h \in \bbz^d$ with norm $j=\|\bs h\|$ is of order $\mathcal{O}(j^{d-1})$ and
\begin{align*}
V_1:=\sum_{\bs h \in \bbz^d} \phi^{\|\bs h\|}=\sum_{0 \leq j < \infty} \phi^j N(j) < \infty.
\end{align*} 
The univariate margins are unit Fr{\'e}chet with scale parameter $V_1$; i.e.,  
\begin{align*}
\mathbb{P}(X^\star(\bs 0) \leq x)=\exp\Big\{-\frac{V_1}{x}\Big\}, \quad x>0.
\end{align*}
Define $Q_{\bs h}(j):=\big|\{\bs s \in \bbz^d: \min\{\|\bs s\|,\|\bs s+\bs h\|\}=j\}\big| \leq 2N(j).$ 
With $V_2(\bs h):=V_2(\bs h;1,1)=\sum_{0 \leq j < \infty} Q_{\bs h}(j)\,\phi^j$, we have for the bivariate margins at lag $\bs h\in\Z^d$,
\begin{align*}
\mathbb{P}(X^\star(\bs 0) \leq x, X^\star(\bs h) \leq x)=\exp\Big\{-\frac{V_2(\bs h)}{x}\Big\}, \quad x > 0.
\end{align*}
We standardise the process \eqref{MMA} by setting 
\begin{align} \label{MMA_stand}
X(\bs s):=X^\star(\bs s)/V_1, \quad \bs s \in \bbz^d.
\end{align}
As a consequence we can choose $a_m=m_n^d$ in Definition~\ref{extremo}. 
We further conclude that the extremal coefficient (cf. \citet{Beirlant}, Section~8.2.7) at lag $\bs h \in \bbz^d$ for the process \eqref{MMA_stand} is given by
\begin{align}\label{extremo_MMA}
\theta(\bs h):=\frac{V_2(\bs h)}{V_1}=\frac{1}{V_1} \sum_{0 \leq j < \infty} Q_{\bs h}(j)\,\phi^j.
\end{align}
Note that $2-\theta(\bs h)=\frac{1}{V_1} \sum_{\|\bs h\|/2 \leq j < \infty} \phi^j[2N(j)-Q_{\bs h}(j)],$ where we use that $Q_{\bs h}(j)=2N(j)$ for $j<\|\bs h\|/2$; see \cite{Cho}, p.~8.

We now verify the conditions of Theorem~\ref{stasyn} for the process $\{X(\bs s) : \bs s\in\Z^d\}$ as in $\eqref{MMA_stand}$ and a chosen set of lags $\mathcal{H} \subseteq B(\bs 0,\ga)$ for some $\ga>0$. \\
We start with an upper bound for the $\alpha$-mixing coefficients $\alpha_{k,\ell}(r)$ for $k, \ell \in \bbn$ and $r>0$ defined in \eqref{alpha_balls}.
To simplify notation, we use $C$ throughout to denote some positive constant, although the actual value of $C$ may change from line to line. 
As in equation~(3.3) of \citet{buhl1} we use Corollary~2.2 of \citet{Dombry} to deduce
\begin{align}\label{4.10a}
\alpha_{k,\ell}(r) &\leq C k\ell\sup_{\|\bs h\|>r} 2-\theta(\bs h) \nonumber\\
&\leq C k\ell\sup_{\|\bs h\|>r}  \sum_{\|\bs h\|/2 \leq j < \infty} \phi^j[2N(j)-Q_{\bs h}(j)] \nonumber \\
& \leq C k\ell\sum_{r/2 \leq j < \infty} 2N(j)\,\phi^j \leq C  k\ell\sum_{r/2 \leq j < \infty} j^{d-1}\, \phi^j,
\end{align}
since $N(j)$ is of order $\mathcal{O}(j^{d-1})$ as mentioned above. 
An integral bound yields for fixed $k,\ell \in \bbn$,
\begin{align}
\alpha_{k,\ell}(r) &\leq C\int_{r/2}^{\infty} t^{d-1}\,\phi^t \, dt=C|\log(\phi)|^{-d} \Gamma\big(d,\frac{r}{2}|\log(\phi)|\big)\\
&\leq C \phi^{r/2} \sum_{k=0}^{d-1} \frac{r^k|\log(\phi)|^k}{2^k k!} =\mathcal{O}(r^{d-1}\,\phi^{r/2}), \label{bound_alpha_MMA}
\end{align}
as $r \to \infty.$ We denote by $\Gamma(s,y) = \int_y^{\infty}t^{s-1}e^{-t} dt = (s-1)!e^{-y}\sum_{i=0}^{s-1}y^i/i!$ for $s \in \mathbb{N}$ the incomplete gamma function.
Since $r^{d-1}\phi^{r/2} \to 0$, this  implies  that $\{X(\bs s) : \bs s\in\Z^d\}$ is $\alpha$-mixing; thus (M1) is satisfied.\\
Now we verify (M3). 
To this end we compute for $\bs s, \bs s' \in \bbz^d$ and $x>0$, using a Taylor expansion, the limit as $x \to \infty$:
\begin{align*}
\mathbb{P}(X(\bs s)>x,X(\bs s')>x)&=1-2\mathbb{P}(X(\bs 0) \leq x)+\mathbb{P}(X(\bs s)\leq x,X(\bs s')\leq x) \\
&=1-2\exp\Big\{-\frac{1}{x}\Big\}+\exp\Big\{-\frac{V_2(\bs s-\bs s')}{V_1x}\Big\} \\
&=\frac{2}{x}-\frac{V_2(\bs s-\bs s')}{V_1 x}+\mathcal{O}\big(\frac{1}{x^2}\big)\\
&=\frac{1}{x}(2-\theta(\bs s-\bs s'))+\mathcal{O}\big(\frac{1}{x^2}\big)\\
&\le \frac{C}{x} \sum_{\|\bs s - \bs s'\|/2\le j <\infty} j^{d-1}\,\phi^j +\mathcal{O}\big(\frac{1}{x^2}\big)
\end{align*}
by \eqref{4.10a}.
Hence,  for $\epsilon>0$, as $\nto$,
\begin{align*}
& \mathbb{P}\Big(\max_{\bs s \in B(\bs 0,\ga)} X(\bs s)> \epsilon m_n^d,\max_{\bs s \in B(\bs h,\ga)} X(\bs s')> \epsilon m_n^d\Big)\\ 
\leq & \sum_{\bs s \in B(\bs 0,\ga)} \sum_{\bs s' \in B(\bs h,\ga)} \mathbb{P}\big(X(\bs s)> \epsilon m_n^d, X(\bs s')> \epsilon m_n^d\big) \\
\leq & \sum_{\bs s \in B(\bs 0,\ga)} \sum_{\bs s' \in B(\bs h,\ga)} \Big\{ \frac{C}{\epsilon m_n^d} \sum_{\|\bs s-\bs s'\|/2 \leq j < \infty} j^{d-1}\,\phi^j +\mathcal{O}\big(\frac{1}{m_n^{2d}}\big)\Big\} \\
\leq & \sum_{\bs s \in B(\bs 0,\ga)} \sum_{\bs s' \in B(\bs h,\ga)} \Big\{ \frac{C}{\epsilon m_n^d} \|\bs s-\bs s'\|^{d-1} \phi^{\frac{\|\bs s-\bs s'\|}{2}}+\mathcal{O}\big(\frac{1}{m_n^{2d}}\big)\Big\} \\
\leq &\frac{C|B(\bs 0, \ga)|^2}{\epsilon m_n^d} (\|\bs h\|-2 \ga)^{d-1} \phi^{\frac{\|\bs h\|-2\ga}{2}} + \mathcal{O}\big(\frac{1}{m_n^{2d}}\big),
\end{align*}
where in the second last step we use the same bound as in \eqref{bound_alpha_MMA}, and in the last step we use that $\|\bs h\|^{d-1} \phi^{\|\bs h\|/2}$ decreases for sufficiently large $\|\bs h\|$.
Therefore, we conclude
\begin{align*}
&\lim\limits_{\kto } \limsup\limits_{\nto }
\sum\limits_{\bs h \in \mathbb{Z}^{d} :  k< \Vert{\bs h}\Vert \leq r_n}
\Big\{ m_n^d 
\mathbb{P}\Big(\max\limits_{\bs s \in B(\bs 0,\gamma)} 
X(\bs s)>\epsilon m_n^d, \max\limits_{\bs s' \in B(\bs h,\gamma)} X(\bs s')>\epsilon m_n^d\Big)\Big\} \\
\leq & \lim\limits_{\kto } \limsup\limits_{\nto } \,C\sum\limits_{\bs h \in \mathbb{Z}^{d} :  k< \|\bs h\| \leq r_n} \Big\{ (\|\bs h\|-2 \ga)^{d-1} \phi^{\frac{\|\bs h\|-2\ga}{2}}\Big\}+ \limsup\limits_{\nto}\,\mathcal{O}\big(\frac{r_n^d}{m_n^d}\big)=0,
\end{align*} 
since $r_n=o(m_n)$,
where we use for the last inequality that $\big|\{\bs h \in \bbz^d: \|\bs h\| \leq r_n\}\big|=\mathcal{O}(r_n^d)$.\\
Turning to condition (M4i), using~\eqref{bound_alpha_MMA}, we have as $\nto$,
\begin{align*}
m_n^d \sum_{\bs h \in \bbz^d: \|\bs h\| > r_n} \alpha_{1,1}(\|\bs h\|) &\leq C m_n^d \sum_{j > r_n} j^{d-1}\, \alpha_{1,1}(j) \\
& \leq C m_n^d \sum_{j > r_n} j^{2(d-1)}\, \phi^{j/2} \leq  C m_n^d  r_n^{2(d-1)} \phi^{r_n/2},
\end{align*}
which follows again from an integral bound. 
Since $$m_n^dr_n^{2(d-1)} \phi^{r_n/2}=\exp\{d\log(m_n)-r_n|\log(\phi^{1/2})|+2(d-1) \log (r_n)\},$$ 
if we choose $m_n$ and $r_n$ such that 
\begin{align}
\log(m_n)=o(r_n),\quad \nto, \label{MMA1}
\end{align}
then condition (M4i) is satisfied.\\
Now observe that for $2 \leq p+q \leq 4$, using again \eqref{bound_alpha_MMA},
\begin{align*}
\sum_{\bs h \in \bbz^d} \alpha_{p,q}(\| \bs h\|) \leq \alpha_{p,q}(0)+ C \sum_{j>0} j^{d-1} \,\alpha_{p,q}(j) \leq \alpha_{p,q}(0) + C \sum_{j>0} j^{2(d-1)}\,\phi^{j/2}  < \infty.
\end{align*}
This shows (M4ii).\\
Finally, we turn to the condition (M4iii) and compute, using \eqref{4.10a} and \eqref{bound_alpha_MMA},
\begin{align*}
m_n^{d/2} n^{d/2} \alpha_{1,n^d}(r_n) &\leq C m_n^{d/2} n^{(3d)/2} r_n^{d-1} \phi^{r_n/2}\\
&=C \exp\big\{\frac{3d}{2}\log(n)-r_n|\log(\phi)|+\frac{d}{2}\log(m_n)+(d-1)\log(r_n)\big\}.
\end{align*}
Thus, we must choose $r_n$ such that
\begin{align}
\log(n)=o(r_n),\quad \nto. \label{MMA2}
\end{align}
To satisfy both \eqref{MMA1} and \eqref{MMA2} and the conditions $r_n=o(m_n)$, $m_n=o(n)$, we can thus choose $m_n=n^{\beta_1}$ and $r_n=n^{\beta_2}$ with $0 < \beta_2 < \beta_1 <1$. 
Hence, Remarks~\ref{rem4.3}(ii) and~\ref{rem4.4}(ii) apply such that \eqref{CLT_Truebeta} holds for $\beta_1 \in (1/3,1/2)$ and $\beta_2 \in (0,\min\{\beta_1;1/2-\beta_1\})$.
\eexam

\bexam[Brown-Resnick process]\label{ex:BR}\\
Consider a strictly stationary {Brown-Resnick process}, which has representation
\begin{align}\label{limitfield}
X(\bs{s}) = \bigvee\limits_{j=1}^\infty \left\{\xi_j \,  e^{W_j(\bs{s})-\delta(\bs s)} \right\},\quad \bs{s}\in\R^d.
\end{align}
where 
$\{\xi_j : j\in\N\}$ are points of a Poisson process on $[0,\infty)$ with intensity $\xi^{-2}d\xi$, the \textit{dependence function} $\delta$ is non-negative and conditionally negative definite and 
$\{W_j(\bs{s}): \bs{s} \in \mathbb{R}^d\}$ are independent replicates of a Gaussian process $\{W(\bs s): \bs{s} \in \mathbb{R}^d\}$ with stationary increments, $W(\bs{0})=0$, $\mathbb{E} [W(\bs{s})]=0$, and covariance function
$$\cov[W(\bs{s}^{(1)}),W(\bs{s}^{(2)})] 
=\delta(\bs{s}^{(1)})+\delta(\bs{s}^{(2)}) -\delta(\bs{s}^{(1)}-\bs{s}^{(2)}).$$
All finite-dimensional distributions are multivariate extreme value distributions with standard unit Fr\'echet margins. 
Representation \eqref{limitfield} goes back to \citet{deHaan} and \citet{Gine}; for more details on Brown-Resnick processes we refer to \citet{Brown}, \citet{Steinkohl}, and \citet{Schlather2}. 
Brown-Resnick processes have been successfully fitted to time series, spatial data and space-time data. Inference methods include both parametric and semi- or non-parametric approaches. 
Empirical studies can for example be found in \citet{buhl1,Steinkohl2,Steinkohl3, Cho,Huser,Ribatet} and \citet{steinkohlphd}.
This important model is treated in detail in the accompanying paper \cite{Steinkohl3}.
In that paper it is proved that the mixing conditions of Theorem~\ref{stasyn} hold for sequences
 $r_n=o(m_n)$, $m_n=o(n)$ and
that  we can choose $m_n=n^{\beta_1}$ and $r_n=n^{\beta_2}$ with $0 < \beta_2 < \beta_1 <1$. 
Hence, Remarks~\ref{rem4.3}(ii) and~\ref{rem4.4}(ii) apply such that \eqref{CLT_Truebeta} holds for $\beta_1 \in (1/3,1/2)$ and $\beta_2 \in (0,\min\{\beta_1;1/2-\beta_1\})$.
Moreover, we prove there that for $\beta_1 \leq 1/3$, the empirical extremogram can be bias-corrected such that the resulting empirical estimator satisfies a CLT to the true extremogram. We further derive a semiparametric estimator for a parametrised extremogram based on a least squares procedure, investigate its behaviour in a simulation study, and apply it to space-time data.
\eexam

\section{Proof of Theorem~\ref{stasyn} }

The empirical extremogram as defined in \eqref{EmpEst} can be viewed as a ratio of estimates of {$\mu_{B(\bs 0, \ga)}(C)$} and {$\mu_{B(\bs 0, \ga)}(D)$} for two suitably chosen sets $C$ and $D$. 
Thus we first derive a LLN and a CLT for such estimates, formulated in the two Lemmas~\ref{st-asy_la} and \ref{asym_normal} below.

We consider estimates of $\mu_{B(\bs 0,\ga)}(C)$, where $C$ is a $\mu_{B(\bs 0,\ga)}$-continuous Borel set in $\ov\R^{|B(\bs 0,\ga)|}\setminus\{\bs 0\}$  (i.e. $\mu_{B(\bs 0, \ga)}(\partial C)=0$). In particular, there exists some $\eps>0$ such that $C\subset\{\bs x\in \R^{|B(\bs 0,\ga)|} : \|\bs x\|>\eps\}$.
In view of \eqref{B6} a natural estimator for $\mu_{B(\bs 0, \ga)}(C)$ is
\beam\label{estmu}
\widehat{\mu}_{B(\bs 0,\ga),m_n}(C) := \Big(\frac{m_n}{n}\Big)^d \sum_{\bs s\in \cals_n} \mathbbmss{1}_{\{\frac{\bs Y(\bs s)}{a_m} \in C\}}.
\eeam
The proof is based on a big block/small block argument as follows. 
We choose sequences $m_n$ and $r_n$ satisfying the conditions of Theorem~\ref{stasyn}, and divide the grid $\cals_n$ into $(n/m_n)^d$ big $d$-dimensional blocks of side length $m_n$, where for simplicity we assume that $n/m_n$ is an integer. 
From those blocks we then cut off smaller blocks, which consist of the first $r_n$ elements in each of the $d$ dimensions. 
The large blocks are then separated (by these small blocks) with at least the distance $r_n$ in all dimensions and shown to be asymptotically independent. 
The construction is an extension of the corresponding time series construction; an interpretation of the big and small blocks in that framework can be found for example in \citet{DMZ} at the end of page~15.

\begin{lemma} \label{st-asy_la}
Let $\{X(\bs s): \bs s \in \mathbb{R}^d\}$ be a strictly stationary regularly varying  process in $\R^d$. 
Let $C$ be some $\mu_{B(\bs 0,\ga)}$-continuous Borel set in $\ov\R^{|B(\bs 0,\ga)|} \setminus \{\bs 0\}$.
Suppose that the following  mixing conditions are satisfied.
\begin{enumerate}[(1)]
\item 
$\{X(\bs{s}): \bs{s}\in\bbr^d\}$ is $\alpha$-mixing with $\alpha$-mixing coefficients $\alpha_{k,l}(r)$ defined in \eqref{alphaBolt}.
\item
There exist sequences $m:=m_n, r:=r_\nto$ with $m_n/n \to 0$ and $r_n/m_n \to 0$ as $\nto$ such that (M3) and (M4i) hold.
\end{enumerate}
Then, as $\nto$,
\begin{align}
\mathbb{E}\big[\widehat{\mu}_{B(\bs 0,\ga),m_n}(C)\big] & \to \mu_{B(\bs 0,\ga)}(C),\label{consist1}\\
\var\big[\widehat{\mu}_{B(\bs 0,\gamma),m_n}(C)\big] 
&\sim \Big(\frac{m_n}{n}\Big)^d \Big(\mu_{B(\bs 0,\ga)}(C)+\sum_{\bs h \in \mathbb{Z}^{d}\setminus\{\bs 0 \}} \tau_{B(\bs 0,\ga)\times B(\bs h,\ga)}(C \times C)\Big) \nonumber \\
&=: \Big(\frac{m_n}{n}\Big)^{d}\sigma_{B(\bs 0,\gamma)}^2(C).\label{varpm2}
\end{align}
If $\mu_{B(\bs 0,\ga)}(C)=0$, \eqref{varpm2} is interpreted as $\var\big[\widehat{\mu}_{B(\bs 0,\gamma),m_n}(C)\big] = o(m_n/n)$.
In particular, we have
\begin{align}\label{lln}
\widehat{\mu}_{B(\bs 0,\ga),m_n}(C) \stp \mu_{B(\bs 0,\ga)}(C), \quad \nto .
\end{align}
\end{lemma}

\bproof
Strict stationarity and relation \eqref{B6} imply that 
\begin{align*}
\mathbb{E}\big[\widehat{\mu}_{B(\bs 0,\ga),m_n}(C)\big]
& =\Big(\frac{m_n}{n}\Big)^d \sum_{\bs s \in \cals_n} \mathbb{P}\Big(\frac{\bs Y(\bs s)}{a_m} \in C\Big) 
= m_n^d\mathbb{P}\Big(\frac{\bs Y(\bs 0)}{a_m} \in C\Big) \\
& \to \mu_{B(\bs 0,\ga)}(C),\quad \nto.
\end{align*}
Further observe that
\begin{align}
\lefteqn{\var\big[\widehat{\mu}_{B(\bs 0,\ga),m_n}(C)\big] 
=\Big(\frac{m_n}{n} \Big)^{2d} \var\big[\sum_{\bs s\in \cals_n} \mathbbmss{1}_{\{\frac{\bs Y(\bs s)}{a_m} \in C\}}\big]}\notag\\
&=\Big(\frac{m_n}{n}\Big)^{2d} \big(n^d\var\big[\mathbbmss{1}_{\{\frac{\bs Y(\bs 0)}{a_m} \in C\}}\big] 
+\sum_{\bs s, \bs s' \in \cals_n \atop \bs s \neq \bs s'} 
\cov\big[\mathbbmss{1}_{\{\frac{\bs Y(\bs s)}{a_m} \in C\}},\mathbbmss{1}_{\{\frac{\bs Y(\bs s')}{a_m} \in C\}}\big]\Big) \notag\\
&=: A_1+A_2. \label{varpm}
\end{align}
By \eqref{B6} and since $\mathbb{P}(\bs Y(\bs 0)/a_m \in C) \to 0$ as $\nto$,
\begin{align*}
A_1&= \Big(\frac{m_n^{2}}{n}\Big)^d\mathbb{P}\Big(\frac{\bs Y(\bs 0)}{a_m} \in C\Big) \Big(1-\mathbb{P}\Big(\frac{\bs Y(\bs 0)}{a_m} \in C\Big)\Big) \sim {\Big(\frac{m_n}{n}\Big)^d} \mu_{B(\bs 0,\ga)}(C) {\to 0}.
\end{align*} 
Let 
$$L=L(n):=\{\bs h=(\bs s-\bs s')\in\Z^d : \bs s, \bs s'\in \cals_n, \bs s\neq \bs s' \}$$ 
be the set of spatial lags on the observation grid. 
We divide the spatial lags in $L$ into different sets. Observe that a spatial lag $\bs h=(h_1,\ldots, h_d)$ appears in $L$ exactly $\prod_{j=1}^d(n-|h_j|)$ times. 
For fixed $k \in\N$ we therefore have by stationarity 
\beam
\Big(\frac{n}{m_n}\Big)^d A_2 &=& 
m_n^d \sum_{\bs h \in L} \prod\limits_{j=1}^d \Big(1-\frac{|h_j|}{n}\Big) \cov[\mathbbmss{1}_{\{\frac{\bs Y(\bs 0)}{a_m} \in C\}},\mathbbmss{1}_{\{\frac{\bs Y(\bs h)}{a_m} \in C\}}] \nonumber\\
&=& m_n^d \Big(\sum_{\bs h \in L \atop 0<\Vert{\bs h}\Vert \leq k} + \sum_{\bs h\in L \atop k<\Vert{\bs h}\Vert \leq r_n}  + \sum_{\bs h\in L \atop \Vert{\bs h}\Vert > r_n}\Big)\nonumber \\
&&\quad\quad\quad\quad \prod\limits_{j=1}^d \Big(1-\frac{|h_j|}{n}\Big)\cov\big[\mathbbmss{1}_{\{\frac{\bs Y(\bs 0)}{a_m} \in C\}},\mathbbmss{1}_{\{\frac{\bs Y(\bs h)}{a_m} \in C\}}\big] \nonumber\\
&=:& A_{21}+A_{22}+A_{23}.\label{cov}
\eeam
Concerning $A_{21}$ we have,  
\begin{align*}
A_{21} = & m_n^d \sum_{\bs h \in L \atop 0 < \Vert{\bs h}\Vert \leq k} \prod\limits_{j=1}^d \Big(1-\frac{|h_j|}{n}\Big)
\Big[\mathbb{P}\Big(\frac{\bs Y(\bs 0)}{a_m} \in C, \frac{\bs Y(\bs h)}{a_m} \in C\Big)- \mathbb{P}\Big(\frac{\bs Y(\bs 0)}{a_m} \in C\Big)^2\Big].
\end{align*}
We have by \eqref{B6},
$$m_n^d \mathbb{P}\Big(\frac{\bs Y(\bs 0)}{a_m} \in C\Big)^2 \sim \mu_{B(\bs 0,\gamma)}(C) \mathbb{P}\Big(\frac{\bs Y(\bs 0)}{a_m} \in C\Big) \to 0,\quad n \to \infty.$$ 
Moreover, for $\bs h \in \mathbb{Z}^{d}\setminus\{\bs 0 \}$, by \eqref{B7},
\beam\label{positive}
m_n^d\mathbb{P}\Big(\frac{\bs Y(\bs 0)}{a_m} \in C, \frac{\bs Y(\bs h)}{a_m} \in C\Big) \to \tau_{{B(\bs 0,\gamma) \times B(\bs h,\gamma)}} (C \times C),\quad n \to \infty.
\eeam
Finally, by dominated convergence, 
\begin{align}\label{dom}
\lim_{\kto } \limsup_{n \to \infty} A_{21} =\sum_{\bs h \in \mathbb{Z}^{d}\setminus\{\bs 0 \}} \tau_{B(\bs{0},\ga)\times B(\bs{h},\ga)}(C \times C).
\end{align}
As to $A_{22}$, observe that for all $n \geq 0$ we have $\prod\limits_{j=1}^d (1-\frac{|h_j|}{n})\leq 1$ for $\bs h\in L$. 
Furthermore, since $C$ is bounded away from $\bs 0$, there exists $\epsilon>0$ such that $C \subset \{\bs x \in \ov{\mathbb{R}}^{|B(\bs 0,\ga)|}: \| \bs x \| > \epsilon\}.$
Hence, we obtain 
\begin{align*}
&|A_{22}| \leq  \sum_{\bs h\in L \atop k < \Vert{\bs h}\Vert \leq r_n} m_n^d\mathbb{P}\big(\frac{\bs Y(\bs 0)}{a_m} \in C, \frac{\bs Y(\bs h)}{a_m} \in C\big) + m_n^d \mathbb{P}\big(\frac{\bs Y(\bs 0)}{a_m} \in C\big)^2\\
 \leq & \sum_{\bs h \in \mathbb{Z}^{d} \atop k < \Vert{\bs h}\Vert \leq r_n} \Big\{m_n^d  \mathbb{P}\big(\|\bs Y(\bs 0)\|>\epsilon a_m, \|\bs Y(\bs h)\|>\epsilon a_m\big) 
+ \frac1{m_n^d} \big(m_n^d \mathbb{P}\big(\frac{\bs Y(\bs 0)}{a_m} \in C\big)\big)^2 \Big\}.
\end{align*}
From \eqref{B6} we know that $m_n^d\mathbb{P}(\bs Y(\bs 0)\in a_m C)\to \mu_{B(\bs 0,\ga)}(C)$ and, hence, the second summand can be estimated by 
${(r_n/m_n)^d\to 0}$ as $\nto$.
The first sum tends to 0 by (M3), exploiting the equivalence of norms on $\mathbb{R}^{|B(\bs 0,\gamma)|}$,
and it follows that
$$\lim_{\kto }\limsup_{\nto } A_{22}=0$$ 
Using the definition of $\alpha$-mixing for $A_1=\{\bs Y(\bs 0)/a_m \in C\}$ and $A_2=\{\bs Y(\bs h)/a_m \in C\}$, we obtain, 
\begin{align}\label{a23}
|A_{23}|&  \leq  m_n^d \sum_{\bs h\in L:  \|\bs h\| > r_n} 
\Big| \mathbb{P}\Big(\frac{\bs Y(\bs 0)}{a_m} \in C, \frac{\bs Y(\bs h)}{a_m} \in C\Big)
- \mathbb{P}\Big(\frac{\bs Y(\bs 0)}{a_m} \in C\Big)^2\Big| \nonumber\\
& \leq  m_n^d \sum_{\bs h \in \mathbb{Z}^d: \|\bs h\| > r_n} \alpha_{1,1}(\|\bs h\|) \to 0,\quad n \to \infty,
\end{align}
by condition (M4i).

Summarising these computations, we obtain from \eqref{cov} and \eqref{dom} that
$$A_2 \sim \Big(\frac{m_n}{n}\Big)^d \sum_{\bs h \in \mathbb{Z}^{d}\setminus\{\bs 0 \}} \tau_{B(\bs 0,\gamma) \times B(\bs h,\gamma)}(C \times C),\quad\nto,$$
 and, therefore, \eqref{varpm} implies \eqref{varpm2}.
Since ${m_n/n\to 0}$ as $\nto$, equations~\eqref{consist1} and \eqref{varpm2} imply \eqref{lln}. 
\eproof

For the proof of the next lemma, in contrast to \citet{Cho}, we proceed similarly as in the proofs of Lemma~\ref{st-asy_la} above and of Theorem~3.2 of \citet{Davis2} and keep the sequence $r_n$ (instead of $m_n$ in \cite{Cho}) in \eqref{ssm} as the distance between the large blocks. 
This construction allows for the much weaker conditions (M2).

\ble\label{asym_normal}
Let $\{X(\bs s): \bs s \in \mathbb{R}^d\}$ be a strictly stationary regularly varying  process in $\R^d$. 
Let the assumptions of Theorem~\ref{stasyn} hold for some $\ga\ge 0$.
Let $C$ be some $\mu_{B(\bs 0,\ga)}$-continuous Borel set in $\ov\R^{|B(\bs 0,\ga)|}\backslash\{\bs 0\}$.
Then
\begin{align}\label{cltpm}
&\wh S_{B(\bs 0,\gamma),m_n}  :=\Big(\frac{m_n}{n}\Big)^{d/2} \sum_{\bs s \in \cals_n} 
\Big[\mathbbmss{1}_{\{\frac{\bs Y(\bs s)}{a_m} \in C\}}-\mathbb{P}\Big(\frac{\bs Y(\bs s)}{a_m} \in C\Big)\Big]\nonumber\\
&=\Big(\frac{n}{m_n}\Big)^{d/2} (\widehat{\mu}_{B(\bs 0,\ga),m_n}(C)-\mu_{B(\bs 0,\ga),m_n}(C)) \std \mathcal{N}(0,\sigma_{B(\bs 0,\gamma)}^2(C)),
\end{align}
as $\nto$ with $\widehat{\mu}_{B(\bs 0,\ga),m_n}(C)$ as in \eqref{estmu}, $\mu_{B(\bs 0,\ga),m_n}(C)):=m_n^d \mathbb{P}(\bs Y(\bs 0)/a_m \in C)$ and $\sigma_{B(\bs 0,\gamma)}^2(C)$ given in~\eqref{varpm2}.
\ele

\bproof
Like \citet{Cho} we follow Lemma~2 in \citet{Bolthausen} and define 
\beam\label{I}
I(\bs s):=\mathbbmss{1}_{\{\bs Y(\bs s)/ a_m \in C\}}-\mathbb{P}(\bs Y(\bs 0)/a_m \in C),\quad \bs s \in \cals_n.
\eeam
Note that by stationarity,
\begin{align}\label{smpm}
\wh S_{B(\bs 0,\gamma),m_n} 
&=\Big(\frac{m_n}{n}\Big)^{d/2} \sum_{\bs s\in \cals_n } I(\bs s).
\end{align}
The boundary condition required in equation~(1) in \citet{Bolthausen} is trivially satisfied for the regular grid $\cals_n$.\\
We first observe that $0 \leq \mathbb{E}[I(\bs s)I(\bs s')]
=\cov[ \mathbbmss{1}_{\{\bs Y(\bs s)/a_m \in C\}},  \mathbbmss{1}_{\{\bs Y(\bs s')/a_m \in C\}}]$
and, using~\eqref{varpm2} for the asymptotic result, that
\begin{align}\label{covfinite}
0 <  \sigma^2_{B(\bs 0,\ga)} & \sim \var[\wh S_{B(\bs 0,\gamma),m_n}], \,
\var[\wh S_{B(\bs 0,\gamma),m_n}] \le \Big(\frac{m_n}{n}\Big)^d \sum_{\bs s,\bs s' \in \bbz^d} |\mathbb{E}[I(\bs s)I(\bs s')]| < \infty, 
\end{align}
such that $\sum_{\bs s,\bs s' \in \bbz^d} \mathbb{E}[I(\bs s)I(\bs s')]>0$.
Finiteness in \eqref{covfinite} follows from a classic result found e.g. in \citet{Ibragimov}, Theorems~17.2.2 and~17.2.3, and the required summability conditions of the $\alpha$-mixing coefficients.\\
Next, we define 
\beam\label{ssm}
 v_n:=\Big(\frac{m_n}{n}\Big)^{d} \ \sum_{\bs s, \bs s' \in \cals_n \atop \|\bs s-\bs s'\| \leq {r_n}}  \mathbb{E}\big[ I(\bs s)  I(\bs s')\big].
\eeam
Decompose
\begin{align}\label{varcov}
v_n &=\Big(\frac{m_n}{n}\Big)^d 
\sum_{\bs s,\bs s'\in \cals_n \atop \|\bs s-\bs s'\| \leq { r_n}} \E[I(\bs s) I(\bs s')]
\nonumber \\
& = \Big(\frac{m_n}{n}\Big)^d \sum_{\bs s, \bs s' \in \cals_n} \E[I(\bs s) I(\bs s')] -\Big(\frac{m_n}{n}\Big)^d 
\sum_{\bs s,\bs s'\in \cals_n \atop \|\bs s-\bs s'\| > r_n} \E[I(\bs s) I(\bs s')] \nonumber\\
& =\var{[\wh S_{B(\bs 0,\gamma),m_n}]}-\Big(\frac{m_n}{n}\Big)^d 
\sum_{\bs s,\bs s'\in \cals_n \atop \|\bs s-\bs s'\| > r_n} \E[I(\bs s) I(\bs s')].
\end{align}
Hence, using the asymptotic result in~\eqref{covfinite},
\begin{align*}
\frac{v_n}{\var{[\wh S_{B(\bs 0,\gamma),m_n}]}}
&=1-\Big(\frac{m_n}{n}\Big)^d \frac1{\si^2_{B(\bs 0,\gamma)}(C)}
\sum\limits_{\bs s,\bs s'\in \cals_n \atop \|\bs s- \bs s'\| > r_n}  \E[I(\bs s) I(\bs s')](1+o(1)).
\end{align*}
Now note that 
\begin{align*}
& \Big(\frac{m_n}{n}\Big)^d 
\sum_{\bs s, \bs s'\in \cals_n \atop \|\bs s- \bs s'\| > r_n} \E[I(\bs s) I(\bs s')] \\
\leq & m_n^d
\sum_{\bs h \in L: \|\bs h\| > r_n} \prod\limits_{j=1}^d \Big(1-\frac{|h_j|}{n}\Big) \\
& \quad\quad\quad \Big|\mathbb{P}\Big(\frac{\bs Y(\bs 0)}{a_m} \in C, \frac{\bs Y(\bs h)}{a_m} \in C \Big) -\Big[\mathbb{P}\Big(\frac{\bs Y(\bs 0)}{a_m} \in C\Big)\Big]^2 \Big| \\
\leq & m_n^d \sum_{\bs h \in \mathbb{Z}^d: \|\bs h\| > r_n} \alpha_{1,1}(\| \bs h \|) \rightarrow 0, \quad \nto,
\end{align*}
as in \eqref{a23}, with $\alpha$-mixing coefficients defined in~\eqref{alpha_balls}.
Therefore, 
\begin{align} \label{preasyvar}
v_n \sim \var[\wh S_{B(\bs 0,\gamma),m_n}] \rightarrow \si^2_{B(\bs 0,\gamma)}(C), \quad \nto.
\end{align}
Next we define the standardized quantities 
\begin{align*}
\overline{S}_{n} & :=v_n^{-1/2}\wh S_{B(\bs 0,\gamma),m_n}=  v_n^{-1/2}\Big(\frac{m_n}{n}\Big)^{d/2} \sum_{\bs s\in \cals_n } I(\bs s),\\
\overline{S}_{\bs s,n} &:=v_n^{-1/2}\Big(\frac{m_n}{n}\Big)^{d/2}  \sum_{\bs s' \in \cals_n \atop \|\bs s- \bs s'\| \leq r_n} I(\bs s').
\end{align*}
We now show condition (b) of Lemma~2 in \citet{Bolthausen}.
To this end let $\im \in \mathbb{C}$ be the complex imaginary unit.
 If $\lim_{\nto } \mathbb{E}\big[(\im\lambda-\overline{S}_n)\exp\{\im\lambda\overline{S}_n\}\big]=0$ for all $\lambda \in \mathbb{R}$, then (by Stein's Lemma) the law of $\overline{S}_n$ converges to the standard normal one and we obtain \eqref{cltpm} by \eqref{smpm} and \eqref{preasyvar}.\\
First note that for arbitrary $\lambda \in \mathbb{R}$,
 \begin{align*}
&(\im\lambda-\overline{S}_n)\exp\{\im \lambda \overline{S}_n\} \\
=& \im\lambda\exp\{\im\lambda \overline{S}_n\}\Big(1-v_n^{-1/2}\sum_{\bs s\in \cals_n} \Big(\frac{m_n}{n} \Big)^{d/2} I(\bs s) \ov S_{\bs s,n}\Big) \\
&-v_n^{-1/2}\exp\{\im\lambda \overline{S}_n\} \sum_{\bs s \in \cals_n} \Big(\frac{m_n}{n}\Big)^{d/2} I(\bs s) (1-\exp\{-\im\lambda \overline{S}_{\bs s,n}\}-\im\lambda\overline{S}_{\bs s,n}) \\
&-v_n^{-1/2}\sum_{\bs s \in \cals_n}\Big(\frac{m_n}{n}\Big)^{d/2} I(\bs s) \exp\{-\im\lambda(\overline{S}_{\bs s,n}-\overline{S}_n)\} \\ 
=:&B_1-B_2-B_3.
\end{align*}
Since $|\exp\{ix\}|=1$ for all $x \in \mathbb{R}$, we compute
$$|B_1|=|\lambda|\Big|1-v_n^{-1} \Big(\frac{m_n}{n}\Big)^d  \sum_{\bs s,\bs s' \in \cals_n  \atop \|\bs s-\bs s'\| \leq {r_n}} I(\bs s)I(\bs s')\Big|$$ 
and, using \eqref{ssm},
\beao
|B_1| &=&  |\lambda| v_n^{-1} 
\Big(\frac{m_n}{n}\Big)^{d} \Big| \sum_{\ \|\bs s-\bs s'\| \leq {r_n}} I(\bs s)I(\bs s')- \sum_{ \|\bs s-\bs s'\| \leq {r_n}}  \mathbb{E}\big[ I(\bs s)  I(\bs s')\big]\Big|\\
&= &  |\lambda| v_n^{-1} 
\Big(\frac{m_n}{n}\Big)^{d}  \Big|\sum_{ \|\bs s-\bs s'\| \leq {r_n}}
\Big( I(\bs s)I(\bs s')-  \mathbb{E}\big[ I(\bs s)  I(\bs s')\big]\Big)\Big|,
\eeao
such that 
\beao
E[|B_1|^2]=\lambda^2 v_n^{-2} \Big(\frac{m_n}{n}\Big)^{2d} \sum_{  \|\bs s-\bs s'\| \leq { r_n}} \sum_{  \|\bs \ell- \bs \ell'\| \leq {r_n}} \cov\big[I(\bs s)I(\bs s'),I(\bs \ell)I(\bs \ell') \big].
\eeao
We use {definition~\eqref{alpha_balls} of the $\alpha$-mixing coefficients} for
$${\Lambda_1'=\{\bs s,\bs s'\}\quad\mbox{and}\quad \Lambda_2'=\{\bs \ell,\bs \ell'\}},$$
then  {$|\Lambda_1'|,|\Lambda_2'| \leq 2$, and
for $d(\Lambda_1',\Lambda_2')$} we consider the following two cases:
\begin{enumerate}[(1)]
\item 
$\|\bs s-\bs \ell\| \geq 3 r_n.$ Then $2r_n \leq (2/3) \|\bs s-\bs \ell\|$ and  {$d(\Lambda_1',\Lambda_2')\ge \|\bs s- \bs \ell\|-2 r_n$}.
Since indicator variables are bounded, by Theorem~17.2.1 of \citet{Ibragimov} we have  
\begin{align*}
|\cov\big[I(\bs s)I(\bs s'),I(\bs \ell)I(\bs \ell') \big]| &\leq 4 \alpha_{2,2}\Big(\|\bs s-\bs \ell\|-2r_n\Big)
\leq 4 \alpha_{2,2}\Big(\frac1{3} \|\bs s-\bs \ell\|\Big).
\end{align*}
 The last inequality holds, since $\alpha_{2,2}$ is a decreasing function. 
\item 
$\|\bs s-\bs \ell\| {< 3 r_n.}$ 
Set 
$j:=\min\{\|\bs s- \bs \ell\|, \|\bs s- \bs \ell'\|,\|\bs s'- \bs \ell\|, \|\bs s'- \bs \ell'\|\}$,
then $d(\Lambda_1',\Lambda_2')\ge j$ and,
again by Theorem~17.2.1 of \cite{Ibragimov}, 
$$\cov\big[I(\bs s)I(\bs s'),I(\bs \ell)I(\bs \ell') \big] \leq 4 \alpha_{p,q}(j), \quad 2 \leq p+q \leq 4.$$ 
\end{enumerate}
Therefore,
\begin{align*}
& E[|B_1|^2] 
\leq  \frac{4\lambda^2}{v_n^{2}} \Big(\frac{m_n}{n}\Big)^{2d}  \\
& \quad\quad\quad\quad  \Big[\sum_{  \|\bs s- \bs \ell\| {\geq 3 r_n}} 
 \sum_{{ \|\bs s- \bs s'\| 
 \leq { r_n}} \atop \|\bs \ell- \bs \ell'\| 
 \leq { r_n} } \alpha_{2,2}\Big(\frac1{3} \|\bs s- \bs \ell\|\Big)
+\sum_{\|\bs s- \bs \ell\| {< 3 r_n}} \sum_{{ \|\bs s- \bs s'\| \leq { r_n}} \atop 
\|\bs \ell- \bs \ell'\| \leq {r_n} } \alpha_{p,q}(j)\Big]\\
& \leq  \frac{4\lambda^2}{v_n^2} \Big(\frac{m_n}{n}\Big)^{2d}  n^d{ r_n}^{2d} \Big[\sum_{ \bs h\in \mathbb{Z}^{d}:\|\bs h\| {\geq 3 r_n}} 
 \alpha_{2,2}\Big(\frac1{3} \|\bs h\|\Big)+\sum_{ \bs h\in \mathbb{Z}^{d}:\|\bs h\| {< 3 r_n}}  \alpha_{p,q} (\|\bs h\|)\Big].
\end{align*}
The last inequality unfolds by stationarity as follows: 
we obtain $n^d$ by summation over all $\bs s\in \cals_n $, whereas $r_n^{2d}$ arises from summation over all $\bs s'$ and $\bs \ell'$ such that $\|\bs s- \bs s'\| \leq { r_n}$ and $\|\bs \ell- \bs \ell'\| \leq { r_n}$, respectively. 
By (M4ii) the sums in brackets are finite and thus 
$$E[|B_1|^2]  = {\mathcal{O}\Big(\Big(\frac{m_n^2 r_n^2}{n}\Big)^d\Big)},$$ 
which converges to $0$ as $\nto$ by (M2).\\
Now we show that $\mathbb{E}[|B_2|] \to 0$ as $\nto .$ 
Since $|1-\exp{\{-\im x\}}-\im x| \leq  x^2/2$ for $x \in \mathbb{R}$ and $|I(\bs s)| \leq 1$ for $\bs s \in \cals_n$, we find
$$|B_2| \leq \frac{1}{2} \lambda^2 v_n^{-1/2} \Big(\frac{m_n}{n}\Big)^{d/2} \sum_{\bs s \in \cals_n} \overline{S}_{\bs s,n}^2.$$
By stationarity and \eqref{ssm} for the second equality below,
\begin{align*}
&\mathbb{E}[|B_2|]
\le  \frac{1}{2} \lambda^2 v_n^{-1/2} \Big(\frac{m_n}{n}\Big)^{d/2} \ n^d \mathbb{E}[\overline{S}_{\bs 0,n}^2] \\
&= \frac{1}{2} \lambda^2 v_n^{-3/2} \Big(\frac{m_n}{n}\Big)^{d/2}  \ m_n^d\sum_{\bs s \in \cals_n:  \|\bs s\| \leq { r_n}} \sum_{\bs s'\in \cals_n: \|\bs s'\| \leq { r_n }} \mathbb{E}[I(\bs s)I(\bs s')] \\
&\leq \frac{1}{2} \lambda^2 v_n^{-3/2}  \Big(\frac{m_n}{n}\Big)^{d/2}\ r_n^d \ m_n^d  
\sum_{\bs s\in \cals_n } \mathbb{E}[I(\bs 0)I(\bs s)] \\
&= \mathcal{O}\Big(\Big(\frac{m_n}{n}\Big)^{d/2} r_n^d\Big) = {\mathcal{O}\Big(\Big(\frac{m_n  r_n^2}{n}\Big)^{d/2}\Big)},
\end{align*}
where we used \eqref{covfinite} to obtain $m_n^d\sum_{\bs s \in \cals_n } \mathbb{E}[I(\bs 0)I(\bs s)]  = \mathcal{O}(1)$.
Again by (M2) we find that  $\mathbb{E}[|B_2|] \rightarrow 0$ as $\nto$.\\
Next we estimate $B_3$:
\begin{align*}
\mathbb{E}[B_3]
&= v_n^{-\frac{1}{2}} \Big(\frac{m_n}{n}\Big)^{d/2} 
\sum\limits_{\bs s \in \mathcal{S}_n} \mathbb{E}\Big[I(\bs s) \exp\big\{-\im\lambda (\overline{S}_{\bs s,n}-\overline{S}_n)\big\}\Big]\\
&= v_n^{-\frac{1}{2}} \Big(\frac{m_n}{n}\Big)^{d/2} 
\sum\limits_{\bs s \in \mathcal{S}_n} \mathbb{E}\Big[I(\bs s) \exp\Big\{\im\lambda v_n^{-\frac{1}{2}} \Big(\frac{m_n}{n}\Big)^{d/2} \sum\limits_{\bs s' \in \cals_n \atop \|\bs s -\bs s'\|> r_n} I(\bs s') \Big\} \Big]\\
&= v_n^{-\frac{1}{2}} m_n^{d/2} n^{d/2} 
 \mathbb{E}\Big[I(\bs 0) \exp\Big\{\im\lambda v_n^{-\frac{1}{2}} \Big(\frac{m_n}{n}\Big)^{d/2} \sum\limits_{\|\bs s \|> r_n} I(\bs s) \Big\} \Big],
\end{align*}
where the last equality holds by stationarity.
{We use definition~\eqref{alpha_balls} of the $\alpha$-mixing coefficients for
$$\Lambda_1'=\{\bs 0\}\quad\mbox{and}\quad \Lambda_2'=\{\bs s \in \cals_n: \|\bs s\| > r_n\},$$
then   $|\Lambda_1'|=1$, $|\Lambda_2'| \leq n^d$ and $d(\Lambda_1',\Lambda_2')>r_n$. }
Abbreviate 
$$\eta(r_n) :=\exp\Big\{\im\lambda v_n^{-\frac{1}{2}} \Big(\frac{m_n}{n}\Big)^{d/2} \sum\limits_{ \|\bs s \|> r_n} I(\bs s) \Big\},$$ 
then $I(\bs 0)$ and $\eta(r_n)$ are measurable with respect to ${\mathcal F}_{\Lambda_1}$ and ${\mathcal F}_{\Lambda_2}$, respectively, {where $\Lambda_i=\cup_{\bs s \in \Lambda_i'} B(\bs s,\ga)$ for $i=1,2.$}
Now we apply  Theorem~17.2.1 of Ibragimov and Linnik to obtain
$$|\mathbb{E}[B_3]|
\leq  {4 v_n^{-1/2}m_n^{d/2} n^{d/2}\alpha_{1, n^d }(r_n) \to 0},$$
where  convergence to 0 is guaranteed by condition (M4iii).
\end{proof}

The proof of Theorem~\ref{stasyn} follows now similarly as that of Corollary~3.4 in \citet{Davis2} (also used in Theorem~1 in \citet{Cho}). In order to keep the paper self-contained, we summarize the main ideas.\\ \\

\noindent
\textbf{Sketch of the proof of Theorem~\ref{stasyn}.}
For Borel sets $C,D \subseteq \ov{\mathbb{R}}^{|B(0,\ga)|} \backslash \{\bs 0\}$  such that $\mu_{B(0,\ga)}(D) > 0$,  define the ratio
$$R_n(C,D):=\frac{\mathbb{P}(\bs Y(\bs 0)/a_m \in C)}{\mathbb{P}(\bs Y(\bs 0)/a_m \in D)}$$
and the correspondent empirical estimator
$$\widehat{R}_n(C,D):=\frac{\sum\limits_{\bs s \in \cals_n} \one{\{ \bs Y(\bs s)/a_m \in C\}}}{\sum\limits_{\bs s \in \cals_n} \one{\{ \bs Y(\bs s)/a_m \in D\}}}.$$
Recall the definition of $\mathcal{H}=\{\bs h_1,\ldots,\bs h_p\}$.
For $1 \leq i \leq p$ fix a  lag $\bs h_i=(h_i^{1},\ldots,h_i^{d}) \in \mathcal{H}$ and denote as before 
$${\cals_n(\bs h_i)=\{\bs s \in \cals_n: \bs s + \bs  h_i \in \cals_n\}\mbox{ with }
|\cals_n(\bs h_i)|}=\prod\limits_{j=1}^d (n-|h_i^{j}|) \sim n^d,\, \nto.$$
Then the empirical extremogram as defined in \eqref{EmpEst} for Borel sets $A,B$ in $\overline{\mathbb{R}}\backslash\{0\}$ satisfies as $\nto$,
\begin{align*}
\widehat{\rho}_{AB,m_n}(\bs h_i)
&\sim \frac{\sum\limits_{\bs s \in \cals_n(\bs h_i)} \mathbbmss{1}_{\{X(\bs s)/a_m \in A, X(\bs s + \bs h_i)/a_m \in B\}}}{\sum\limits_{\bs s \in \cals_n} \one{\{ X(\bs s)/a_m \in A\}}} \\
& \sim \frac{\sum\limits_{\bs s \in \cals_n} \one{\{\bs Y(\bs s)/a_m \in D_{i}\}}}{\sum\limits_{\bs s \in \cals_n} \one{\{\bs Y(\bs s)/a_m \in D_{p+1}\}}}=\widehat{R}_n(D_i,D_{p+1}),
\end{align*}
by definition \eqref{Di} of the sets $D_i$ for $i=1,\ldots,p$.
Moreover, the pre-asymptotic extremogram defined in \eqref{preasymptotic} can be written as
\begin{align*}
\rho_{AB,m_n}(\bs h_i)&=\frac{\mathbb{P}(X(\bs 0)/a_m \in A,X(\bs h_i)/a_m \in B)}{\mathbb{P}(X(\bs 0)/a_m \in A)} =\frac{\mathbb{P}(\bs Y(\bs 0)/a_m \in D_i)}{\mathbb{P}(\bs Y(\bs 0)/a_m \in D_{p+1})}\\
&=R_n(D_i,D_{p+1}).
\end{align*}
This implies that proving \eqref{asyspace} requires a central limit theorem for the scaled vector of ratio differences  
\begin{align}\label{Rn}
\Big(\frac{n}{m_n}\Big)^{d/2} \big[\widehat{R}_n(D_i,D_{p+1})-R_n(D_i,D_{p+1})\big]_{i=1,\ldots,p}.
\end{align}
Now observe that for fixed $i \in \{1,\ldots,p\}$,
\beao
\lefteqn{\widehat{R}_n(D_i,D_{p+1})-R_n(D_i,D_{p+1})
 = \frac{\widehat{\mu}_{B(\bs 0, \ga),m_n}(D_i)}{\widehat{\mu}_{B(\bs 0, \ga),m_n}(D_{p+1})}-\frac{\mu_{B(\bs 0,\gamma),m_n}(D_i)}{\mu_{B(\bs 0,\gamma),m_n}(D_{p+1})}}\\
&=&\frac{\mu_{B(\bs 0, \ga),m_n}(D_{p+1})/\widehat{\mu}_{B(\bs 0, \ga),m_n}(D_{p+1})}{(\mu_{B(\bs 0, \ga),m_n}(D_{p+1}))^2} \\
&& \times \Big[\widehat{\mu}_{B(\bs 0, \ga),m_n}(D_i)\mu_{B(\bs 0, \ga),m_n}(D_{p+1})-\widehat{\mu}_{B(\bs 0, \ga),m_n}(D_{p+1})\mu_{B(\bs 0, \ga),m_n}(D_i) \Big] \\
&=&\frac{\mu_{B(\bs 0, \ga),m_n}(D_{p+1})/\widehat{\mu}_{B(\bs 0, \ga),m_n}(D_{p+1})}{(\mu_{B(\bs 0, \ga),m_n}(D_{p+1}))^2} \\
& & \quad\quad  \times \Big[\big(\widehat{\mu}_{B(\bs 0, \ga),m_n}(D_i)-\mu_{B(\bs 0, \ga),m_n}(D_i)\big) \mu_{B(\bs 0, \ga),m_n}(D_{p+1})\\
&& \quad\quad\quad -\big(\widehat{\mu}_{B(\bs 0, \ga),m_n}(D_{p+1})-\mu_{B(\bs 0, \ga),m_n}(D_{p+1})\big) \mu_{B(\bs 0, \ga),m_n}(D_i) \Big]\\
& = &  \frac{1+o_p(1)}{(\mu_{B(\bs 0, \ga)}(D_{p+1}))^2} \\
& & \quad\quad  \times \Big[\big(\widehat{\mu}_{B(\bs 0, \ga),m_n}(D_i)-\mu_{B(\bs 0, \ga),m_n}(D_i)\big) \mu_{B(\bs 0, \ga)}(D_{p+1})\\
&& \quad\quad\quad -\big(\widehat{\mu}_{B(\bs 0, \ga),m_n}(D_{p+1})-\mu_{B(\bs 0, \ga),m_n}(D_{p+1})\big) \mu_{B(\bs 0, \ga)}(D_i) \Big]
\eeao
by \eqref{B6},  Lemma~\ref{st-asy_la} and Slutzky's lemma.
For the vector in \eqref{Rn}, recalling that $\mu_{B(\bs 0, \ga)}(D_{p+1})=\mu(A)$, and $F \in \mathbb{R}^{(p+1) \times (p+1)}$ as given in \eqref{matrixF}, we find
\begin{align*}
&\Big(\frac{n}{m_n}\Big)^{d/2} \big[\widehat{R}_n(D_i,D_{p+1})-R_n(D_i,D_{p+1})\big]_{i=1,\ldots,p}\\
=&\Big(\frac{n}{m_n}\Big)^{d/2} \frac{1+o_p(1)}{(\mu_{B(\bs 0, \ga)}(D_{p+1}))^2} \,  F \,
\begin{pmatrix}&\widehat{\mu}_{B(\bs 0, \ga),m_n}(D_1)-\mu_{B(\bs 0, \ga),m_n}(D_1) \\
& \vdots \\ & \widehat{\mu}_{B(\bs 0, \ga),m_n}(D_p)-\mu_{B(\bs 0, \ga),m_n}(D_p) \\ & \widehat{\mu}_{B(\bs 0, \ga),m_n}(D_{p+1})-\mu_{B(\bs 0, \ga),m_n}(D_{p+1})
\end{pmatrix} \\
=:& \Big(\frac{n}{m_n}\Big)^{d/2} \frac{1+o_p(1)}{(\mu_{B(\bs 0, \ga)}(D_{p+1}))^2} \, F \, \bs{\mu}_{m_n}.
\end{align*} 
Thus, it remains to prove that
\begin{align}
\Big(\frac{n}{m_n}\Big)^{d/2} \bs{\mu}_{m_n} \std \mathcal{N}(\bs 0, \Sigma), 
\label{ratio_asyn}
\end{align}
where $\Sigma$ is given in the statement of the Theorem.
This can be done as in \citet{Davis2}, Corollary~3.3 using the Cram{\'e}r-Wold device and similar ideas as in the proofs of Lemmas~\ref{st-asy_la} and~\ref{asym_normal}. 
In particular, note that for all $i,j \in \{1,\ldots,p+1\}$ as $\nto$,
\begin{align*}
&\cov[\wh \mu_{B(\bs 0, \ga),m_n}(D_i),\wh \mu_{B(\bs 0, \ga),m_n}(D_j)] \\
&\quad\quad\quad \sim \frac{m_n}{n^d}\Big(\mu_{B(\bs 0, \ga)}(D_i \cap D_j) + \sum_{\bs h\in\Z^d\setminus\{ \bs 0\}} \tau_{B(\bs 0, \ga) \times B(\bs h, \ga)} (D_i \times D_j) \Big).
\end{align*}
\halmos

\appendix
\section{Taylor expansion for the pre-asymptotic extremogram}

\begin{lemma}
\label{LemmaExtremo_Fr}
Let the assumptions of Theorem~\ref{CLT_True} hold. \\
(a) \, For $\bs h \in \mathbb{R}^{d}$ the true extremogram is given by
\begin{align}\label{extremo_Fr}
\rho_{AB}(\bs h)=&\frac{a_1a_2}{a_2-a_1}\Big(-V_2(a_2,b_2)+V_2(a_2,b_1)+V_2(a_1,b_2)-V_2(a_1,b_1)\Big), 
\end{align}
where $V_2(\cdot,\cdot)=V_2(\bs h;\cdot,\cdot)$ is the bivariate exponent measure (cf. \citet{Beirlant}, Section~8.2.2) defined by
$$\mathbb{P}(X(\bs 0) \leq x_1,X(\bs h) \leq x_2)=\exp\{-V_2(x_1,x_2)\}, \quad x_1,x_2>0.$$
For $A=(a,\infty)$ and $B=(b,\infty)$ we obtain
\beam\label{extremo_Fr2}
\rho_{AB}(\bs h) = a \Big(\frac1{a}+\frac1{b} - V_2(a,b)\Big).
\eeam
(b) \, For fixed $\bs h \in \mathbb{R}^{d}$ and the sequence $m_n$ satisfying the conditions of Theorem~\ref{stasyn}, the pre-asymptotic extremogram satisfies as $n \rightarrow \infty$, 
\begin{align}\label{preasym_Fr}
&\rho_{AB,m_n}(\bs h) 
= (1+o(1))\Big[\rho_{AB}(\bs h)+ \\
&\quad \frac1{2m_n^{d}}\frac{a_1a_2}{a_2-a_1}\Big( V_2^2(a_2,b_2) + V_2^2(a_2,b_1) + V_2^2(a_1,b_2) +V_2^2(a_1,b_1)\Big)\Big]. \nonumber
\end{align}
For $A=(a,\infty)$ and $B=(b,\infty)$ this reduces to
\begin{align}\label{preasym_Fr2}
\rho_{AB,m_n}(\bs h)=&(1+o(1))\Big[\rho_{AB}(\bs h) 
+ \frac1{2m_n^{d}a} (\rho_{AB}(\bs h)-2\frac{a}{b})(\rho_{AB}(\bs h)-1)\Big].
\end{align}
\end{lemma}

\bproof
Throughout the proof all asymptotic results hold as $\nto$. 
Since $\{X(\bs s: \bs s \in \bbr^d\}$ has standard unit Fr{\'e}chet margins, we can and do choose $a_n=n^d$ in \eqref{regvar} such that $\mathbb{P}(X(\bs 0) > n^d)=1-\exp\{-n^{-d}\} \sim n^{-d}$.\\
(a) \, We first show \eqref{extremo_Fr}.
With this choice of $a_n$, equation~\eqref{extremo} is equivalent to
\begin{align*}
\rho_{AB}(\bs h)&=\lim\limits_{n \rightarrow \infty}\frac{n^{d}\mathbb{P}(X(\bs 0) \in n^{d}A, X(\bs h) \in n^{d}B)}{n^{d}\mathbb{P}(X(\bs 0) \in n^{d}A)}.
\end{align*}
We set $A=(a_1,a_2)$ and $B=(b_1,b_2)$.
For the denominator we obtain by a first order Taylor expansion
\begin{align}\label{den}
n^d\mathbb{P}(X(\bs 0) \in n^d (a_1,a_2))
= & n^d[\mathbb{P}(X(\bs 0) \leq n^da_2)-\mathbb{P}(X(\bs 0) \leq n^da_1)] \notag\\
= & n^d\big[\exp\{-\frac{1}{n^da_2}\}-\exp\{-\frac{1}{n^da_1}\}\big]\\ 
= & \frac{1}{a_1}-\frac{1}{a_2}+\mathcal{O}(n^{-d}) \rightarrow \frac{1}{a_1}-\frac{1}{a_2}=\frac{a_2-a_1}{a_1a_2}>0. \notag
\end{align}
Since by homogeneity $V_2(kx_1,kx_2))=k^{-1}V_2(x_1,x_2)$ for $k>0$, we find for the numerator 
\begin{align}\label{num}
& n^d\mathbb{P}((X(\bs 0), X(\bs h)) \in n^d (a_1,a_2) \times n^d(b_1,b_2))\notag\\
=&n^d\Bigg[\exp\Big\{-\frac{1}{n^d}V_2(a_2,b_2)\Big\}
-\exp\Big\{-\frac{1}{n^d}V_2(a_2,b_1)\Big\} \nonumber \\
&-\exp\Big\{-\frac{1}{n^d}V_2(a_1,b_2)\Big\}
+\exp\Big\{-\frac{1}{n^d}V_2(a_1,b_1)\Big\}\Bigg]\\
=& -V_2(a_2,b_2)+V_2(a_2,b_1)+V_2(a_1,b_2)-V_2(a_1,b_1) +\mathcal{O}(n^{-d}). \nonumber 
\end{align}
This yields~\eqref{extremo_Fr}. \\
Furthermore, $V_2(a,\infty)=1/a$, $V_2(\infty,b)=1/b$ and $V_2(\infty,\infty)=0$, see for instance \citet{Resnick4}, p.~268. Together with the fact that the denominator converges to $1/a$, this gives \eqref{extremo_Fr2}.\\[2mm]
(b) \, For an estimate of the pre-asymptotic extremogram we need to improve the first order asymptotics of part (a).
For an interval $(a,b)$ we abbreviate $\Phi_n(a,b):=\exp\{-\frac{1}{m_n^d}V_2(a,b)\}$.
From equation~\eqref{preasymptotic} together with \eqref{den} and \eqref{num}  we obtain
\beam\label{pre_asy_general}
\rho_{AB,m_n}(\bs h)
&=&\frac{\mathbb{P}(X(\bs 0) \in m_n^{d}A, X(\bs h) \in m_n^{d}B)}{\mathbb{P}(X(\bs 0) \in m_n^{d}A)}\notag\\
&=& \frac{\Phi_n(a_2,b_2)
-\Phi_n(a_2,b_1)
-\Phi_n(a_1,b_2)
+\Phi_n(a_1,b_1)}{\exp\{-\frac{1}{a_2\,m_n^d}\}-\exp\{-\frac1{a_1\,m_n^d}\}}\notag\\
&=& \rho_{AB}(\bs h) + \frac1{\exp\{-\frac{1}{a_2\,m_n^d}\}-\exp\{-\frac1{a_1\,m_n^d}\}}\notag\\
&& \Big[\Phi_n(a_2,b_2)
-\Phi_n(a_2,b_1)
-\Phi_n(a_1,b_2)
+\Phi_n(a_1,b_1)\notag\\
&&-\Big(\exp\Big\{-\frac{1}{a_2\,m_n^d}\Big\}-\exp\Big\{-\frac1{a_1\,m_n^d}\Big\}\Big)\rho_{AB}(\bs h)
\Big] \notag\\
&=& \rho_{AB}(\bs h) + \frac{a_1a_2}{a_2-a_1} m_n^d(1+o(1))\notag\\
&& \Big[\Phi_n(a_2,b_2)
-\Phi_n(a_2,b_1)
-\Phi_n(a_1,b_2)
+\Phi_n(a_1,b_1)\notag\\
&&-\Big(\exp\Big\{-\frac{1}{a_2\,m_n^d}\Big\}-\exp\Big\{-\frac1{a_1\,m_n^d}\Big\}\Big)\rho_{AB}(\bs h)
\Big] 
\eeam
By a second order Taylor expansion of $\Phi_n$ it follows that, using~\eqref{extremo_Fr} and \eqref{den}, 
\begin{align}
\rho_{AB,m_n}(\bs h) =& (1+o(1))\Big[\rho_{AB}(\bs h)+\notag\\
&  \frac1{2m_n^{d}}\frac{a_1a_2}{a_2-a_1} \Big( V_2^2(a_2,b_2) + V_2^2(a_2,b_1) + V_2^2(a_1,b_2) +V_2^2(a_1,b_1)\Big) 
\Big] \notag.
\end{align}
This shows \eqref{preasym_Fr}. \\
Now let $A=(a,\infty)$ and $B=(b,\infty)$. Then $a_1a_2/(a_2-a_1)=a_1+o(1)$ as $a_2\to\infty$
and the expression in the rectangular bracket in \eqref{pre_asy_general} becomes
\beam \label{rect_bracket}
\Big[\cdots\Big] &=& 
1-\exp\Big\{-\frac1{b\,m_n^d}\Big\}-\exp\Big\{-\frac1{a\,m_n^d}\Big\}+\exp\Big\{-\frac1{m_n^d} V_2(a,b)\Big\} \nonumber \\
&& \quad\quad -\Big(1-\exp\Big\{-\frac1{a\,m_n^d}\Big\}\Big)\rho_{AB}(\bs h).
\eeam
Abbreviating $V_2:=V_2(a,b)$, a second order Taylor expansion gives with \eqref{extremo_Fr2} for the right-hand side of \eqref{rect_bracket},
\beam\label{a8}
&& \Big(\frac1{a\,m_n^d}-\frac1{2a^2\,m_n^{2d}}\Big)+\Big(\frac1{b\,m_n^d}-\frac1{2b^2\,m_n^{2d}}\Big) -\Big(\frac1{m_n^d} V_2- \frac1{2m_n^{2d}} V_2^2\Big)\nonumber\\
&& -\Big(\frac1{m_n^d}-\frac1{2 a\,m_n^{2d}}\Big) \Big(\frac1{a}+\frac1{b} - V_2\Big) +o(m_n^{-2d})\nonumber\\
&=&   \frac1{2\,m_n^{2d}}\Big\{\Big(V_2^2-\frac1{a^2} - \frac1{b^2} \Big) +\frac1{a} \Big(\frac1{a}+\frac1{b} - V_2\Big)\Big\} +o(m_n^{-2d}).
\eeam
Solving \eqref{extremo_Fr2} for $V_2$ gives $V_2=V_2(a,b)=\frac1{a}(1-\rho_{AB}(\bs h))+\frac{1}{b}$ such that we obtain for the expression in the curly brackets of \eqref{a8},
\beao
\big(\frac1{a}(1-\rho_{AB}(\bs h))+\frac{1}{b}\big)^2 -\frac1{a^2} - \frac1{b^2} +\frac1{a^2}\rho_{AB}(\bs h)
&=& \frac1{a^2} (\rho_{AB}(\bs h)-2\frac{a}{b})(\rho_{AB}(\bs h)-1).
\eeao
Going backwards with this expression proves \eqref{preasym_Fr2}.
\eproof

\subsubsection*{Acknowledgements}
\noindent
We thank Erwin Bolthausen and Thomas Mikosch for enlightening discussions.
SB acknowledges support from the Deutsche Forschungsgemeinschaft (DFG) through the TUM  International Graduate School of Science and Engineering (IGSSE).


\end{document}